\documentclass[a4paper,11pt]{article}
\RequirePackage[T1]{fontenc}
\usepackage{palatino}
\usepackage{hyperref}
\hypersetup{colorlinks=true, linkcolor=blue,  anchorcolor=blue, citecolor=blue, filecolor=blue, menucolor=blue, urlcolor=blue}
\usepackage[english]{babel}
\usepackage[utf8]{inputenc}
\usepackage[font={small,sf},labelfont=bf]{caption}
\usepackage{amsmath,amssymb,amsfonts,mathrsfs,amsthm}
\usepackage{graphicx,pdfpages,enumitem}
\theoremstyle{plain}
\newtheorem{theorem}{Theorem}[section] 
 
\newtheorem{proposition}[theorem]{Proposition} 
\theoremstyle{plain} 
\newtheorem{definition}{Definition}[section] 
\theoremstyle{definition} 
\newtheorem{remark}{Remark}[section] 
\usepackage{geometry}
\title{How do Conservative Backbone Curves Perturb into Forced Responses? A Melnikov Function Analysis} 
\author{Mattia Cenedese and George Haller
\thanks{Corresponding author: \href{mailto:georgehaller@ethz.ch}{georgehaller@ethz.ch}}\vspace{0.35cm}\\ 
Institute for Mechanical Systems, ETH Z\"urich,
\\ Leonhardstrasse 21, 8092 Z\"urich, Switzerland}
\date{}
\begin{document}
\maketitle
\begin{abstract}
\noindent Weakly damped mechanical systems under small periodic forcing tend to exhibit periodic response in a close vicinity of certain periodic orbits of their conservative limit. Specifically, amplitude frequency plots for the conservative limit have often been noted, both numerically and experimentally, to serve as backbone curves for the near resonance peaks of the forced response. In other cases, such a relationship between the unforced and forced response was not observed. Here we provide a systematic mathematical analysis that predicts which members of conservative periodic orbit families will serve as backbone curves for the forced-damped response. We also obtain mathematical conditions under which approximate numerical and experimental approaches, such as energy balance and force appropriation, are justifiable. Finally, we derive analytic criteria for the birth of isolated response branches (isolas) whose identification is otherwise challenging from numerical continuation. 
\end{abstract}
\section{Introduction}
Conservative families of periodic orbits, broadly known as nonlinear normal modes (NNMs) in the field of structural engineering, often appear to shape the behaviour of mechanical systems even in the presence of additional damping and time-periodic forcing. Not only has this influence been noted for low-amplitude vibrations of a small number of coupled oscillators, but it also appears to hold for large amplitude motion in arbitrary degrees of freedom. Various descriptions of this phenomenon are available in the literature, ranging from the original introduction of NNMs by Rosenberg \cite{Rosenberg1962} to the more recent reviews lead by Vakakis \cite{Vakakis1997,Vakakis2001,Vakakis2008}, Avramov and Mikhlin \cite{Avramov2011,Avramov2013} and Kerschen \cite{Kerschen2014}.
\begin{figure}[t]
\centering
\includegraphics[width=1\textwidth]{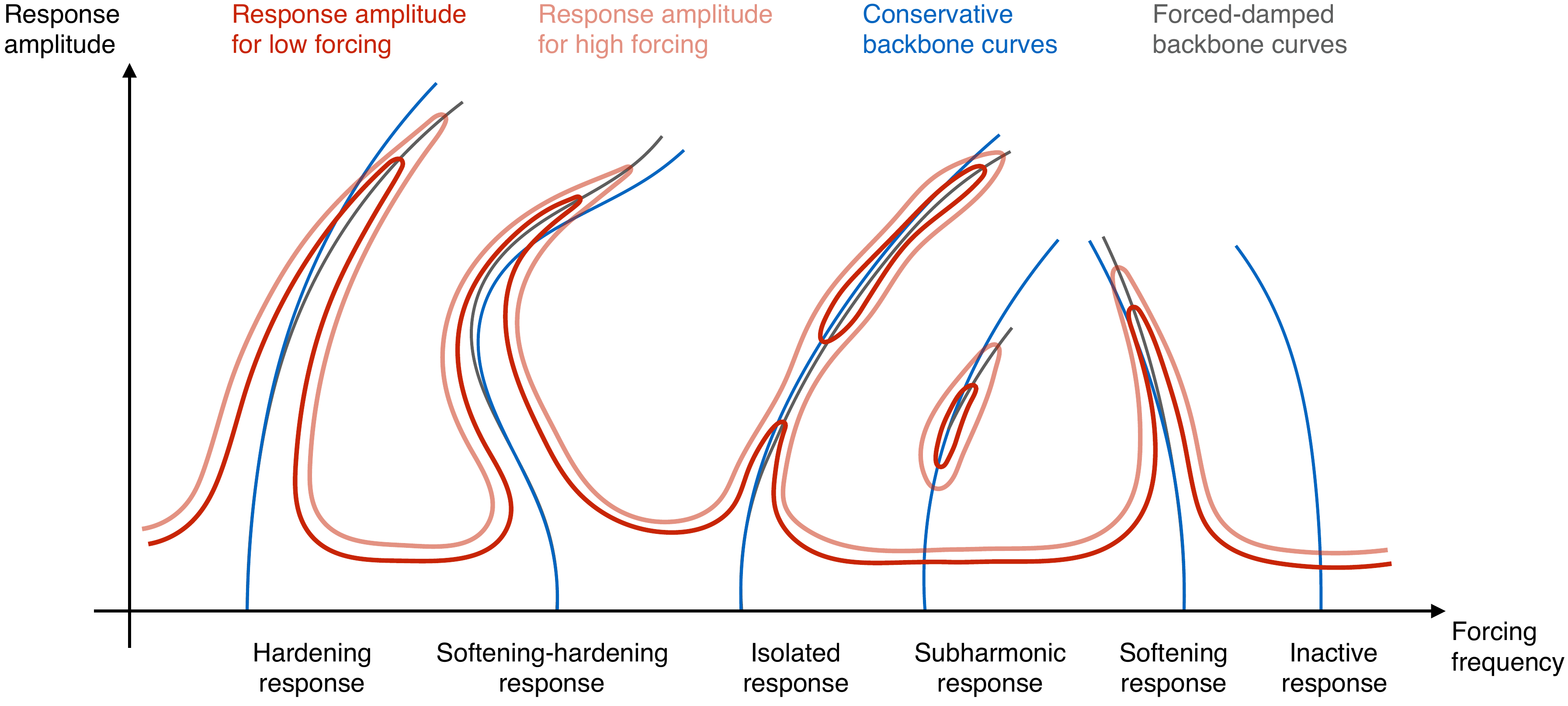}
\caption{Illustration of frequency response phenomena in mechanical systems. The dark and light red curves identify the frequency response for low and high forcing amplitudes, respectively, while blues curves depict conservative backbone curves and grey curves represent forced-damped backbone curves.}
\label{fig:S1_IM1}
\end{figure}
These studies suggest that forced-damped frequency responses might bifurcate from conservative NNMs. To summarise features of such bifurcations, Fig. \ref{fig:S1_IM1} qualitatively represents possible nonlinear phenomena in the frequency response. Dark and light red curves show steady-state solutions for a periodically forced-damped mechanical system for low and high forcing amplitudes, respectively. Blue curves correspond to amplitude-frequency relations of periodic orbit families of the underlying conservative system, which we refer to as conservative backbone curves. Grey curves depict backbone curves of the forced-damped response, defined as the frequency locations of amplitude maxima under variation of the forcing amplitude \cite{NayfehM2007}. The first and last peaks in Fig. \ref{fig:S1_IM1} show the classic hardening and softening resonance trends respectively. As most frequently observed behaviours, these two phenomena have been broadly studied: see, e.g. \cite{NayfehM2007,Kerschen2014} for analytical and numerical treatments and \cite{Peeters2011b,Szalai2017} for experimental results. In these settings, as the response amplitude increases, the backbone curves of the conservative limit and those of the forced-damped system have been noted to pull apart \cite{Touze2006}. 

The backbone curves of the second peak from the left in Fig. \ref{fig:S1_IM1} feature a non-monotonic trend in frequency, i.e., softening for lower amplitudes and hardening for higher ones. This type of behaviour is relevant for shallow-arch systems \cite{Sombroek2018}, MEMS devices \cite{Polunin2016} and structural elements \cite{Carpineto2014}, while a simple mechanical example is available at \cite{Mojahed2018,Liu2019}. The third peak of Fig. \ref{fig:S1_IM1} shows an isolated response curve (isola) that may occur due to the influence of nonlinear damping \cite{Habib2018} or symmetry breaking mechanisms \cite{Hill2016}. In the former case, the isola can join with the main branch when the forcing amplitude exceeds a certain threshold \cite{Ponsioen2019}, as indicated by the red light curve. Subharmonic responses, displayed along the second blue line from the right, also show up as isolated branches, as they cannot originate from the linear limit \cite{NayfehM2007}. Finally, as highlighted in \cite{Hill2017} with the analysis of a simple model of nonlinear beam, not all NNMs contribute to shaping the forced-response, as illustrated by the rightmost backbone curve in Fig. \ref{fig:S1_IM1}.

Analytic relations between conservative families of periodic orbits and frequency responses are only available for specific, low-dimensional oscillators from perturbation expansions that assume the conservative periodic limit to have small amplitudes. These expansions may arise from the method of multiple scales \cite{NayfehM2007}, averaging \cite{Sanders2007}, the first-order normal form technique \cite{Touze2004} or the second-order normal form technique \cite{Neild2011}. Based on this latter method, Hill \textit{et al.} \cite{Hill2015} developed an energy-transfer formulation for locating maxima of the frequency response along conservative backbone curves. However, the authors a priori postulate a relation between conservative oscillations and frequency responses, and also discuss potential limitations arising from this assumption in \cite{Hill2015}. Vakakis and Blanchard in \cite{Vakakis2018} show the exact steady states of a strongly nonlinear periodically forced and damped Duffing oscillator. They also clarify how these forced-damped periodic motions are related the conservative backbone curve, but they restrict the discussion to specific types of periodic forcing.

Relying on exact mathematical results, spectral submanifold (SSM) theory \cite{Haller2016,Breunung2018} has been developed for the local analysis of damped, nonlinear oscillators. This approach is insensitive to the number of degrees of freedom thanks to the automated procedure developed in \cite{Ponsioen2018} and can be hence used for exact nonlinear model reduction. SSMs, however, do not exist in the limit of zero forcing and damping, in which case they are replaced by Lyapunov subcenter manifolds (LSMs) \cite{Lyapunov1992}. The relationship between the dynamics on damped, unforced SSMs and LSMs is now established in a small enough neighbourhood of the unforced equilibrium point \cite{Breunung2018,Ponsioen2019}.

Even though diverse numerical options are available to explore forced responses, analytical tools applicable to arbitrary degrees of freedom and motion amplitudes are still particularly valuable. Not only can such tools help with the analysis of several perturbation types by relying only on the knowledge of conservative orbits, but they can also overcome limitations of numerical routines. For a thorough review of these limitations, we refer the reader to \cite{Renson2016a}. For example, direct numerical integration needs long computational time for high degrees of freedom systems with small damping and it is limited to stable periodic orbits. Despite being very accurate, numerical continuation (shooting methods \cite{Peeters2009}, harmonic balance \cite{Grolet2012} or collocation \cite{Dankowicz2013}) suffers from the curse of dimensionality. Furthermore, it can efficiently compute the main branch of the forced response for frequencies away from resonance, but fails to find isolated branches when their existence and location is not a priori known.

In a parallel development, SSMs, backbone curves, main and isolated branches can now be computed efficiently for general, forced-damped, multi-degree-of-freedom mechanical systems up to any required order of accuracy \cite{Ponsioen2019}. This approach also yields analytic approximations for backbone curves and isolas, as long as these stay within the domain of convergence of Taylor expansions constructed for the forced-damped SSMs \cite{Ponsioen2019}.

Beyond these numerical approaches, experimental methods would also be aided by a rigorous mathematical relation between conservative backbone curves and forced-damped responses. One of such methods has been developed by Peeters \textit{et al.} \cite{Peeters2011a,Peeters2011b}, who propose an extension of the phase-lag quadrature criterion, well-known for linear systems, to nonlinear systems in order to isolate conservative NNMs experimentally. Their method assumes that the nonlinear periodic motion is synchronous \cite{Rosenberg1962}, the damping is linear (or at least odd in the velocity) and the excitation is multi-harmonic and distributed. The idea is to subject the system to forcing that exactly balances damping along a periodic orbit of the conservative limit. Peeters \textit{et al.} find that such a balance holds approximately when each harmonic of the conservative periodic orbit has a phase lag of $90^\circ$ relative to the corresponding forcing one. Force appropriation \cite{Ehrhardt2016,Peter2018} and control-based continuation \cite{Renson2016b} exploit the phase-lag quadrature criterion for systematic tracking of backbone curve. 

Another related experimental method in need of a mathematical justification is resonance decay \cite{Peeters2011a}, which uses a force appropriation routine to isolate a periodic orbit, then turns off the forcing and assumes the response to converge to the equilibrium position along a two-dimensional SSM, sometimes called a damped-NNM \cite{Kerschen2009}. This technique is expected to provide an approximation of conservative backbone curves, but it remains partially unjustified for two reasons. On the analytical side, it assumes a yet unproven purely, parasitic effect of damping on the response. On the experimental side, decoupling the shaker from the system remains a challenge that affects the accuracy of the results. 

Despite available experimental and numerical observations, it is unclear if and when conservative NNMs perturb into forced-damped periodic responses. This is because periodic orbits in conservative systems are never structurally stable under generic perturbations, which tend to destroy them \cite{GH1983}. Indeed, any conservative periodic orbit has at least two Floquet multipliers equal to $+1$ due to the conservation of energy \cite{MO2017}, rendering the orbit structurally unstable. Classic analytic approaches \cite{Malkin1949,Loud1959,Farkas1994} to generic, non-autonomous perturbations of normally hyperbolic periodic orbits are therefore inapplicable in this setting. Conservative NNMs exist in families that are only guaranteed to persist under small, smooth conservative perturbations \cite{MO2017,MoserZehnder2005}.

In its simplest form, the study of dissipative perturbations on a conservative family of periodic orbits dates back to Poincaré \cite{Poincare1892}, developed further by Arnold \cite{Arnold1964}. An important contribution was made by Melnikov \cite{Melnikov1963}, who focused on dissipative perturbations of planar Hamiltonian systems. His approach reduces the persistence problem of periodic orbits to the analysis of the zeros of the \textit{subharmonic Melnikov function} \cite{GH1983,Yagasaki1996}. As extensions of Melnikov's approach, studies on two-degree-of-freedom Hamiltonian systems are available: \cite{Veerman1985,Veerman1986} consider the fate of periodic orbit families in an integrable system subject to Hamiltonian perturbations, while \cite{Yagasaki1999} analyses two fully decoupled oscillators under generic dissipative perturbations. Subharmonic Melnikov-type analysis for non-smooth systems is also available; see \cite{Kunze2000} for examples of planar oscillators and \cite{Shaw1989a,Shaw1989b} for a system with two degrees of freedom. All these results, therefore, require low-dimensionality and integrability before perturbation, neither of which is the case for the conservative limits of nonlinear structural vibrations problems arising in practice. 

As an alternative to these analytic methods, Chicone \cite{Chicone1994,Chicone1995,ChiconeR2000} established a perturbation method for manifolds of isochronous periodic orbits without any restriction on their Floquet multipliers or assumptions on integrability/coupling before perturbation. This elegant approach exploits the Lyapunov-Schmidt reduction to obtain a generic multi-dimensional bifurcation function for the persistence of single orbits. Furthermore, this method has also been extended for non-smooth (but Lipschitz) dynamical systems \cite{Buica2012}. However, these results are not directly applicable to the typical setting of nonlinear structural vibrations. Moreover, an exact resonance condition is required in Chicone's method, even though perturbed periodic orbits are often observed when the forcing frequency clocks in near-resonance with the frequency of a periodic orbit of the conservative limit.

In this paper, we develop an exact analytical criterion for the perturbation of conservative NNMs into forced-damped periodic responses, thereby predicting the variety of behaviours depicted in Fig. \ref{fig:S1_IM1}. We assume that the conservative limit of the system has a one-parameter family of periodic orbits, satisfying generic nondegeneracy conditions. We then study the persistence or bifurcation of these periodic orbits under small damping and time-periodic forcing. Utilising ideas from Rhouma and Chicone \cite{ChiconeR2000}, we reduce this perturbation problem to the study of the zeros of a Melnikov-type function, generalising therefore the original Melnikov method to multi-degree-of-freedom systems. Our approach relies on the smallness of dissipative and forcing terms which is generally satisfied in structural dynamics, but we will not assume that the unperturbed periodic orbit has small amplitude. This distinguishes our approach from various classic perturbation expansions that assume closeness from the unforced equilibrium. Our analysis also differ from classic Melnikov-type approaches in that it does not require the conservative limit of the system to be integrable. Indeed, for our unperturbed conservative limit, we only require the existence of a generic family of periodic orbits that may only be known from numerical continuation. 

When our Melnikov-type method is applied to mechanical systems, it provides a rigorous justification for the classic energy principle, a broadly used but heuristic necessary condition for the existence of periodic response in forced-damped nonlinear oscillations \cite{Hill2016,Hill2015,Peeters2011a}. Our analysis shows that under further conditions, the energy principle becomes a rigorous sufficient condition for nonlinear periodic response and extends to arbitrary number of degrees of freedom, multi-harmonic forcing, large-amplitude periodic orbits and higher-order external resonances.

We first give a mathematical formulation for general dynamical systems, then apply it to the classic setting of nonlinear structural vibrations. In that context, our results reveal how the near-resonance part of the main and isolated branches of the periodic response diagram are born out of the conservative backbone curve. We also discuss how our results justify the phase-lag quadrature criterion under more general conditions than prior studies assumed. Finally, we illustrate the power of our analytic predictions on a six-degree-of-freedom mechanical system.

\section{Setup}
\label{sec:S2}
We consider a mechanical system with $N$ degrees of freedom and denote its generalised coordinates by $q\in U\subset \mathbb{R}^N$, $N\geq 1$. We assume that this system is a small perturbation of a conservative limit that conserves the total energy $H:U\times\mathbb{R}^{N} \rightarrow\mathbb{R}$, defined as
\begin{equation}
H(q,\dot{q})=E(q,\dot{q})+V(q)=\frac{1}{2}\langle\, \dot{q} \, , \, M(q)\dot{q} \rangle + V(q) .
\end{equation}
Here, $M(q)$ is the positive definite, symmetric mass matrix, $E(q,\dot{q})$ is the kinetic energy and $V(q)$ the potential. The equations of motion for the system take the form
\begin{equation}
\label{eq:MechSys}
M(q)\ddot{q}+G(q,\dot{q})+DV(q)=\varepsilon Q(q,\dot{q},t;T,\varepsilon),
\end{equation}
where $\varepsilon\geq 0$ is the perturbation parameter, $G(q,\dot{q})=D_t(M(q))\dot{q}-D_qE(q,\dot{q})$ contains inertial forces and $\varepsilon Q(q,\dot{q},t;T,\varepsilon)=\varepsilon Q(q,\dot{q},t+T;T,\varepsilon)$ denotes a small perturbation of time-period $T$. System in Eq. \ref{eq:MechSys} is then a weakly non-conservative mechanical system.

Introducing the notation $x=(q,\dot{q})\in\mathbb{R}^n$ with $n=2N$, the equivalent first-order form reads
\begin{equation}
\label{eq:NAutSysG}
\dot{x}=f(x)+\varepsilon g(x,t;T,\varepsilon) ,
\end{equation}
where we assume that $f\in C^r$ with $r\geq 2$, while $g$ is $C^{r-1}$ in $t$ and $C^{r}$ with respect to the other arguments. These vector fields are defined as
\begin{equation}
\label{eq:FGDef}
\begin{array}{lcr}
\displaystyle f(x)= \begin{pmatrix} \dot{q} \\ -M^{-1}(q)(DV(q)+G(q,\dot{q})) \end{pmatrix}, & &
\displaystyle g(x,t;T,\varepsilon)= \begin{pmatrix} 0 \\ M^{-1}(q)Q(q,\dot{q},t;T,\varepsilon)  \end{pmatrix}.
\end{array}
\end{equation}
We assume any further parameter dependence in our upcoming derivations to be of class $C^r$. Trajectories of (\ref{eq:NAutSysG}) that start from $\xi\in\mathbb{R}^n$ at $t=0$ will be denoted with $x(t;\xi,T,\varepsilon)=(q(t;\xi,T,\varepsilon),\dot{q}(t;\xi,T,\varepsilon))$. We will also use the shorthand notation $x_0(t;\xi)=(q_0(t;\xi,),\dot{q}_0(t;\xi,))=x(t;\xi,T,0)$ for trajectories of the unperturbed (conservative) limit of system (\ref{eq:NAutSysG}). We recall that, for $\varepsilon=0$, energy is conserved, i.e.,  $H(x_0(t;\xi))=H(\xi)$ holds as long as $x_0(t;\xi) \in U$.

\section{Non-autonomous resonant perturbation of normal families of conservative periodic orbits}
\label{sec:S3}
In this section, we first state our main mathematical results for single conservative orbits then for orbit families. We also discuss the physical meaning of these results in the context of mechanical systems. 

For the $\varepsilon=0$ limit of system (\ref{eq:NAutSysG}), we assume that there exists a periodic orbit $\mathcal{Z}\subset U$ of minimal period $\tau>0$ and we denote by $\Pi(p)$ the monodromy matrix\footnote{The monodromy matrix, or linearised period-$\tau$ mapping, $\Pi(p):\mathbb{R}^n\rightarrow\mathbb{R}^n$ is defined as $\Pi(p)=Y(\tau;p)$ where $Y$ solves the equation of variations \cite{Chicone2000}
	\begin{equation*}
	 \dot{Y}=Df(x_0(t;p))Y, \,\,\,\,\,\,\,\,\,\,\, Y(0)=I.
	\end{equation*} } 
based at any point $p\in\mathcal{Z}$. We consider $m\in\mathbb{N}^+$ multiples of the period and let $\mu_{a,m}$ denote the algebraic multiplicity of the $+1$ eigenvalue of $\Pi^m(p)$ and $\mu_{g,m}$ denote its geometric multiplicity. Note that these two multiplicities are invariant under translations along the orbit, while they may change for different values of $m$. We will need the following definition from \cite{Sepulchre1997}.
\begin{definition}
\label{def:mNPO}
A conservative periodic orbit $\mathcal{Z}$ is $m$-\textit{normal} if one of the following holds:
	\begin{enumerate}[label=\textit{(\alph*)}]
	\item $\mu_{g,m}=1$ ;
	\item $\mu_{g,m}=2$ and $f(p)\notin $ range$(\Pi^m(p)-I)$ .	
\end{enumerate}
\end{definition}
This normality is a nondegeneracy condition under which a one-parameter family, $\mathcal{P}$, of $m$-normal periodic solutions of the vector field $f$ emanates from $\mathcal{Z}$ (see Theorem 4 of \cite{Sepulchre1997} or Theorem 7 of \cite{Doedel2003}). We denote with $\lambda\in\mathbb{R}$ the parameter identifying each individual orbit in the family $\mathcal{P}$.
\begin{figure}[t]
\centering
\includegraphics[width=1\textwidth]{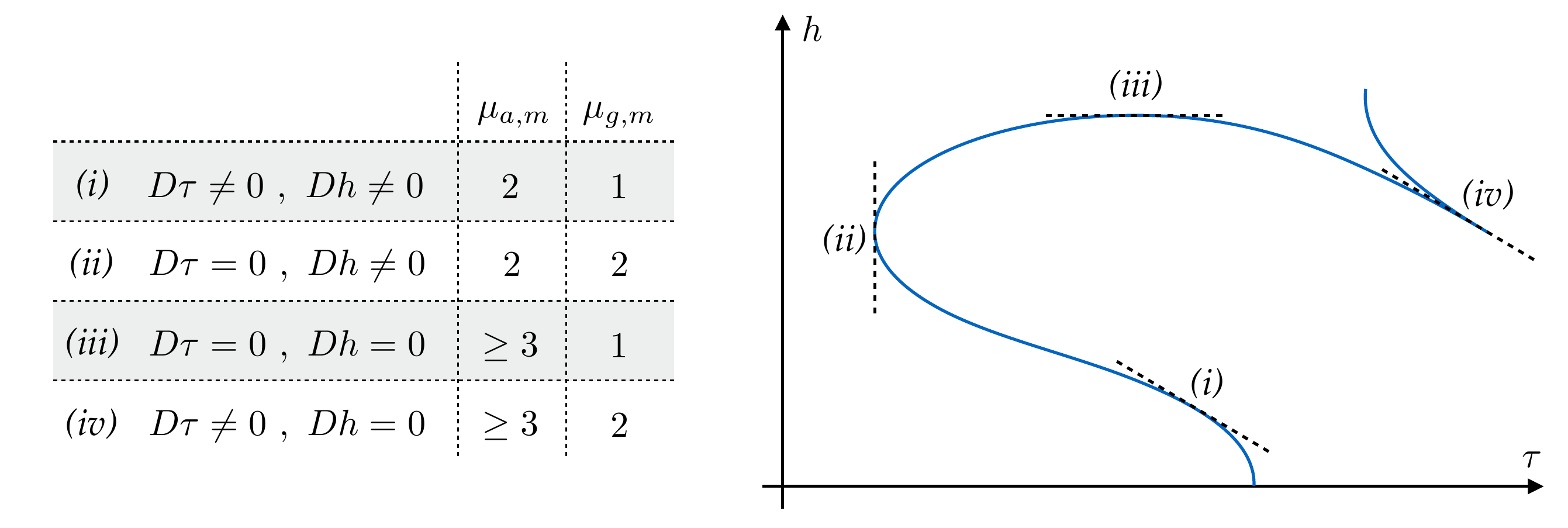}
\caption{Different types of $m$-normal periodic orbits and the associated geometry of the backbone curve, i.e., the relation between the energy $h$ of the periodic response the period $\tau$ of the response.}
\label{fig:S3_IM1}
\end{figure}

Figure \ref{fig:S3_IM1} describes the types of $m$-normal periodic orbits covered by Definition $\ref{def:mNPO}$, with their associated backbone-curve geometry, as given in Theorem 5 of \cite{Sepulchre1997}. The backbone curve can be parametrised as $(\tau(\lambda),h(\lambda))$, with $\tau$ denoting the orbit period and $h$ the value of the first integral. The value of the parameter $\lambda$ is given by a scalar mapping $\lambda=L(\xi,\tau)$ depending on the initial condition $\xi\in\mathbb{R}^n$ and the period $\tau\in\mathbb{R}^+$. We also require $L$ to be invariant under translations of $\xi$ along the orbit. For an $m$-normal periodic orbit belonging to case \textit{(a)} in Definition \ref{def:mNPO}, one can simply choose $L(\xi,\tau)=\tau$. Instead, when $\mu_{a,m}=2$ (see \textit{(i)} and \textit{(ii)} in Fig. \ref{fig:S3_IM1}), the orbit family can be locally parametrised with the value of the first integral $h$, i.e., $L(\xi,\tau)=H(\xi)$. Other possible parametrisations include the value of a coordinate determined by a Poincaré section, the $L^2$ norm of the trajectory or the maximum value of a coordinate along the trajectory. For continuation through cusp points of backbone curves, i.e., \textit{(iv)} in Fig. \ref{fig:S3_IM1}, $L$ may be chosen to provide a pre-defined relation between energy and period (see \cite{Sepulchre1997}), but that is outside the scope of this paper.

\subsection{Perturbation of a single orbit}
\label{sec:S3_1}
Our starting point in the analysis of the fate of perturbed periodic orbits is the displacement map
\begin{equation}
\label{eq:DisplMap}
\Delta_{l}:\mathbb{R}^{n+2}\rightarrow\mathbb{R}^n,\,\,\,\,\,\,\,\,\,\,\, \Delta_{l}(\xi,T,\varepsilon)=x(lT;\xi,T,\varepsilon)-\xi ,\,\,\,\,\,\,\,\,\,\,\,\Delta_{l}\in C^r,
\end{equation}
whose zeros correspond to $lT$-periodic orbits for system (\ref{eq:NAutSysG}) for $l\in\mathbb{N}^+$. We aim to smoothly continue normal periodic orbits in the family $\mathcal{P}$ that exists at $\varepsilon=0$. We consider an $m$-normal periodic orbit $\mathcal{Z}\subset\mathcal{P}$ and assume that $l$ and $m$ are relatively prime integers, i.e., $1$ is their only common divisor. We then look for zeros of (\ref{eq:DisplMap}) that can be expressed as
\begin{equation}
\label{eq:orbeclosex}
\xi=x_0(s;p)+O(\varepsilon),\,\,\,\,\,\,\,\,\,\,\, p\in\mathcal{Z},\,\,\,\,\,\,\,\,\,\,\, T=\tau m/l+O(\varepsilon),
\end{equation}
under the additional constraint
\begin{equation}
\label{eq:orbeclose}
L(\xi,lT)=L(p,m\tau).
\end{equation}
Equation (\ref{eq:orbeclose}) represents a resonance condition as it relates, either explicitly or implicitly, the periods of the perturbation with that of the periodic orbit $\mathcal{Z}$. With the notation $L(p,m\tau)=\lambda$, the zero problem to be solved reads
\begin{equation}
\label{eq:DisplMapC}
\Delta_{l,L}:\mathbb{R}^{n+2}\rightarrow\mathbb{R}^{n+1},\,\,\,\,\,\,\,\,\,\,\, \Delta_{l,L}(\xi,T,\varepsilon)=\begin{cases}
\Delta_{l}(\xi,T,\varepsilon)\\
L(\xi,lT)-\lambda
\end{cases}, \,\,\,\,\,\,\,\,\,\,\,\Delta_{l,L}(\xi,T,\varepsilon)=0 .
\end{equation}
Defining the smooth, $L$-independent, $m\tau$-periodic function $M^{m:l}:\mathbb{R}\rightarrow\mathbb{R}$ as
\begin{equation}
\label{eq:MelFun}
\displaystyle M^{m:l}(s)=\int_0^{m\tau} \big\langle\, DH(x_0(t+s;p))\, , 
\, g(x_0(t+s;p),t;\tau m/l,0)\,\big\rangle\, dt ,
\end{equation}
we obtain the following main result.
\begin{theorem}
	\label{thm:stpert}
	If the Melnikov function $M^{m:l}(s)$ has a simple zero at $s_0\in\mathbb{R}$, i.e.,
	\begin{equation}
	\label{eq:transvzero}
	\begin{array}{lcr}
	M^{m:l}(s_0)=0, & \,\,\,\,& DM^{m:l}(s_0)\neq 0,
	\end{array}
	\end{equation}
	then the $m$-normal periodic orbit $\mathcal{Z}$ of the $\varepsilon=0$ limit smoothly persists in system (\ref{eq:NAutSysG}) for small $\varepsilon>0$. Moreover, in this case, there exists at least another topologically transverse zero in the interval $(s_0,\, s_0+m\tau)$. If $M^{m:l}(s)$ remains bounded away from zero, then $\mathcal{Z}$ does not smoothly persists for small $\varepsilon>0$.
	\end{theorem}
\begin{figure}[t]
\centering
\includegraphics[width=1\textwidth]{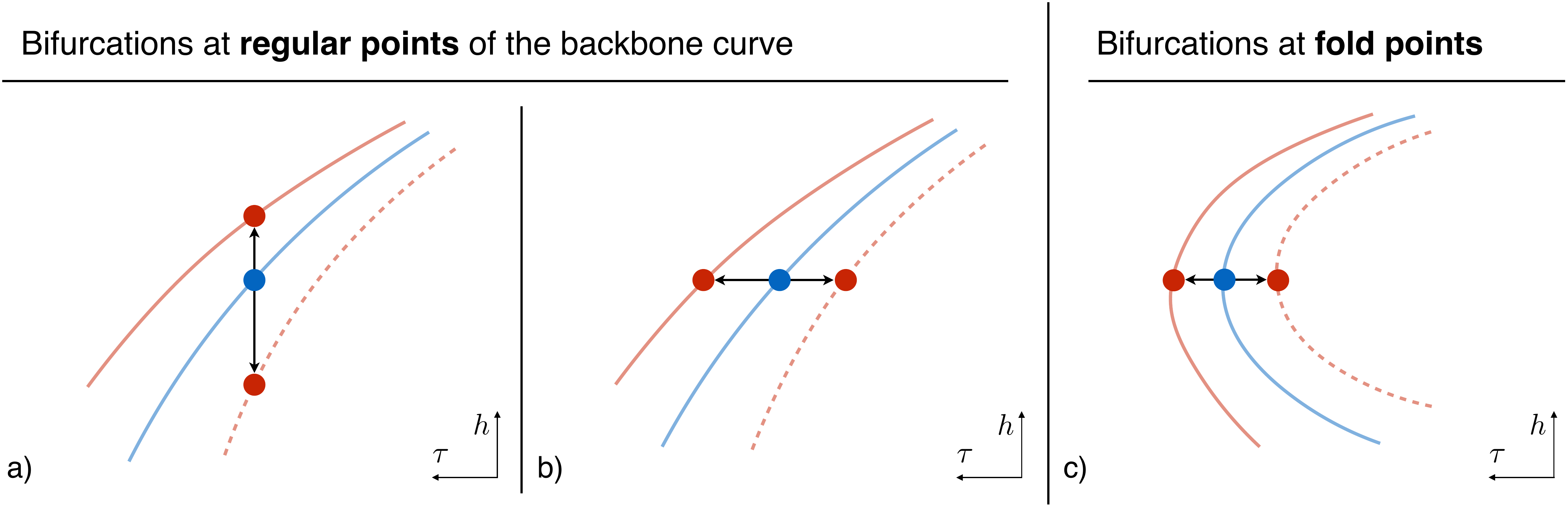}
\caption{Bifurcations in case the Melnikov function (\ref{eq:MelFun}) has two simple zeros. Regular points of the backbone curve generate perturbed solutions either in the isochronous (a) or isoenergetic (b) directions. In contrast, in case (c), perturbed solutions are guaranteed to exist in the isoenergetic direction for a fold point in $\tau$. Blue lines identify conservative backbone curves while red lines mark perturbed periodic orbits. Solid and dashed lines identify different local branches of solutions.}
\label{fig:S3_IM2}
\end{figure}

We prove Theorem \ref{thm:stpert} in Appendix \ref{app:A1}. The proof reduces the $(n+1)$-dimensional persistence problem to the analysis of the zeros of the scalar function (\ref{eq:MelFun}). This function formally agrees with the one derived originally by Melnikov for a planar oscillator, but the proof for $n>2$ is more involved compared to the simple geometric treatment in \cite{GH1983} for $n=2$. 

When the Melnikov function has a simple or transverse zero, the perturbed orbit emanating from the $m$-normal periodic orbit $\mathcal{Z}$ and its period are $O(\varepsilon)$-close to $\mathcal{Z}$ and to $\tau m$, respectively. Since topologically transverse zeros of functions are generically simple, we expect from Theorem \ref{thm:stpert} that an even number of perturbed periodic orbits bifurcates from the $m$-normal periodic orbit at the $\varepsilon=0$ limit, as indeed typically observed in literature. Moreover, since the Melnikov function (\ref{eq:MelFun}) does not depend on the parametrisation function $L$ used in Eq. (\ref{eq:orbeclose}), Theorem \ref{thm:stpert} and its consequences hold for any possible parametrising direction used for the unperturbed periodic orbit family.

Theorem \ref{thm:stpert} can be interpreted directly in terms of the backbone curve of $\mathcal{P}$ and the frequency response of system (\ref{eq:NAutSysG}).  Suppose that the backbone curve of $\mathcal{P}$ shows only regular points near the $m$-normal periodic orbit $\mathcal{Z}$ so that we can select $L(p,m\tau)=m\tau=\lambda$. In this case, Eq. (\ref{eq:orbeclose}) imposes the exact resonance condition $lT=m\tau$. For a pair of simple zeros of $M^{m:l}$, Theorem \ref{thm:stpert} guarantees that the point in the backbone curve corresponding to $\mathcal{Z}$ bifurcates in two frequency responses along the isochronous direction, as depicted in Fig. \ref{fig:S3_IM2}(a). If, instead, $\mathcal{Z}$ corresponds to a fold point in $\tau$, then Theorem \ref{thm:stpert} does not hold for this choice of $L$, but we can still use the energy $h$ as parametrisation variable. In that case, our perturbation method constrains the perturbed initial condition $\xi$ to lie in the same energy level as that of $\mathcal{Z}$. At the same time, the time period $T$ for the perturbed orbit is $O(\varepsilon)$-close to $\tau m/l$, i.e., a near-resonance condition is satisfied. For two simple zeros of $M^{m:l}$, $\mathcal{Z}$ can be smoothly continued in two frequency responses along the isoenergetic direction, as shown in Fig. \ref{fig:S3_IM2}(b) and \ref{fig:S3_IM2}(c).  
\begin{remark}
\label{rmk:icandstab}
While for the classic planar oscillator case the Melnikov function is also able to predict the stability of perturbed orbits \cite{Yagasaki1996}, the stability analysis of persisting periodic orbits is more involved for $n>2$. Indeed, their stability depends on all the Floquet multipliers of the conservative limit, as well as on the nature of the perturbation.
\end{remark}

\subsection{Perturbation of a family and parameter continuation}
\label{sec:S3_2}
Here we consider an additional parameter $\kappa\in\mathbb{R}$ in Eq. (\ref{eq:DisplMapC}), where $\kappa$ is either a feature of the vector fields in system (\ref{eq:NAutSysG}) or the family parameter $\lambda$. The Melnikov function $M^{m:l}$ in (\ref{eq:MelFun}) clearly inherits this smooth parameter dependence.

Next we investigate the fate of the $m$-normal periodic orbit $\mathcal{Z}$ in the family $\mathcal{P}$ for which the Melnikov function features a quadratic zero at $(s_0,\kappa_0)$ defined as:
\begin{equation}
\label{eq:qtanzero}
\begin{array}{lcr}
\displaystyle M^{m:l}(s_0,\kappa_0)=D_{s} M^{m:l}(s_0,\kappa_0)=0, & \,\,\,\,& \displaystyle D^2_{ss} M^{m:l}(s_0,\kappa_0)\neq 0.
\end{array}
\end{equation}
The following theorem describes what generic bifurcations may arise in this setting.
\begin{theorem}
\label{thm:ftpert}
Assume that $M^{m:l}(s,\kappa)$ has a quadratic zero at $(s_0,\kappa_0)$, as defined in Eq. (\ref{eq:qtanzero}). If $D_{\kappa} M^{m:l}(s_0,\kappa_0)\neq 0$, then $\kappa_{sn}=\kappa_0+O(\varepsilon)$ is a bifurcation value at which a saddle-node bifurcation of periodic orbits occurs. If $D_{\kappa} M^{m:l}(s_0,\kappa_0)=0$ and $\mathrm{det}(D^2M^{m:l}(s_0,\kappa_0))>0$ (resp. $<0$), then isola births (resp. simple bifurcations) arise for small $\varepsilon>0$.
\end{theorem}
We prove Theorem \ref{thm:ftpert} in Appendix \ref{app:A1}. Note that the bifurcations described in the last sentence of Theorem \ref{thm:ftpert} are singular ones. Under these, the local, qualitative behaviour of the solutions of Eq. (\ref{eq:DisplMapC}) may change for different small $\varepsilon>0$ as we describe below in an example. On the other hand, periodic orbits arising from either simple zeros or quadratic and $\kappa$-nondegenerate zeros persist for small $\varepsilon>0$. We refer the reader to \cite{ChowHale1982,GolubitskySchaeffer1985,Govaerts2000} for detailed analyses of such singular bifurcations.
\begin{figure}[t]
\centering
\includegraphics[width=1\textwidth]{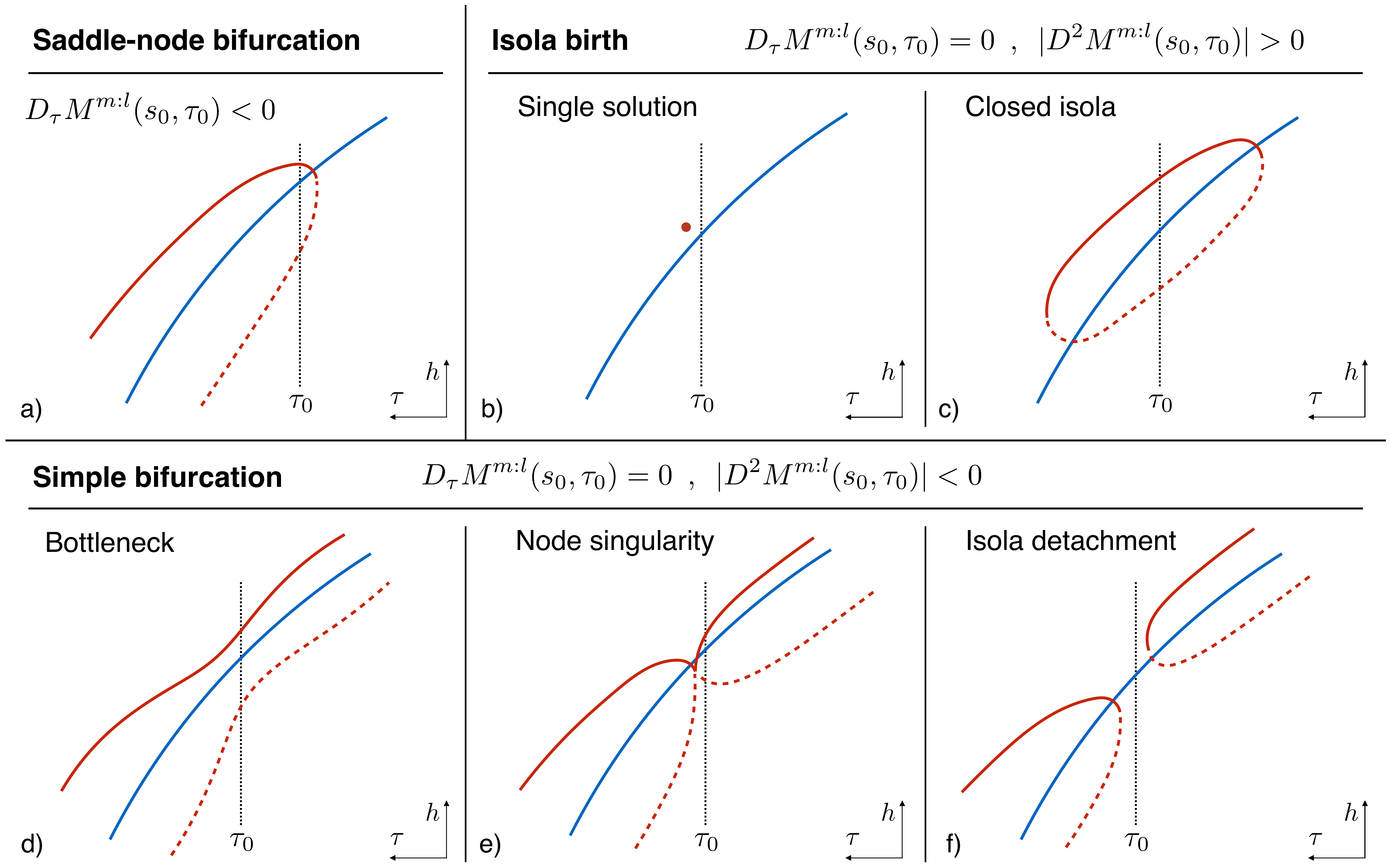}
\caption{Illustration of the bifurcation phenomena described in Theorem \ref{thm:stpert} along a $\tau$-parametrised conservative backbone curve close to a quadratic zero of the Melnikov function. Blue lines identify conservative backbone curves while red lines mark perturbed periodic orbits. Solid and dashed lines identify different local branches of solutions.}
\label{fig:S3_IM3}
\end{figure}

Figure \ref{fig:S3_IM3} illustrates the bifurcations described in Theorem \ref{thm:ftpert} when the family $\mathcal{P}$ can be locally parametrised with the period, which is also the selected continuation parameter, $\kappa=\tau$. This means sweeping along orbits of the family, indicated with a blue line, and analysing when these orbits give rise to perturbed ones in the frequency response, denoted in red.

Plot (a) shows a saddle-node bifurcation in $\tau$, also known as \textit{limit} or \textit{fold} bifurcation. For this type of quadratic zero, the conservative orbit at $\tau_0$ and the period $\tau_0$ itself are $O(\varepsilon)$-close to a locally unique saddle-node periodic orbit in $\tau$ of the frequency response. This unique orbit originates as two solutions branches of Eq. (\ref{eq:DisplMapC}) join together. After this point, conservative orbits of $\mathcal{P}$ do not smoothly persist, at least locally.

The singular case of isola birth, illustrated in Fig. \ref{fig:S3_IM3}(b) and Fig. \ref{fig:S3_IM3}(c), has three possible bifurcation outcomes, depending on the value of the parameters. It is typically observed that either no solution persists from the ones in $\mathcal{P}$ (not show in Fig. \ref{fig:S3_IM3}) or a closed branch of solutions appears, i.e., an \textit{isola} as shown in Fig. \ref{fig:S3_IM3}(c). Instead, the \textit{single solution} case of Fig. \ref{fig:S3_IM3}(b) may occur, but it is non-generic.

Similarly, simple bifurcations may manifest themselves in three scenarios. The \textit{bottleneck}, in Fig. \ref{fig:S3_IM3}(d), and the \textit{isola detachment}, in (f), are generic, while the \textit{node singularity} of Fig. \ref{fig:S3_IM3}(e) is an extreme case.
\begin{remark}
\label{rmk:generalform}
The results we have presented in Theorems \ref{thm:stpert}-\ref{thm:ftpert} apply to general, non-autonomous perturbations of conservative systems with a normal family of periodic orbits, not just to mechanical systems. Moreover, the perturbation may be also of type $g(x,\dot{x},t;T,\varepsilon)$.
\end{remark}

\subsection{The Melnikov function for mechanical systems}
\label{sec:S3_3}
The underlying physics of the full system in Eq. (\ref{eq:MechSys}) implies that any periodic solution with minimal period $lT$ must necessarily experience zero energy balance in one oscillation cycle. Defining the energy function
\begin{equation}
\label{eq:EDef}
E_b(\xi,\varepsilon)_{[0,lT]}=\varepsilon \int_0^{lT} \big\langle\, \dot{q}(t;\xi,T,\varepsilon)\, ,
\, Q(q(t;\xi,T,\varepsilon),\dot{q}(t;\xi,T,\varepsilon),t;T, \varepsilon)\big\rangle dt ,
\end{equation} 
we deduce that, along such a periodic orbit, we must have
\begin{equation}
\label{eq:EBal}
E_b(\xi,\varepsilon)_{[0,lT]}=0 ,
\end{equation} 
given that the work done by the non-conservative forces must vanish over one cycle of that orbit. Equation (\ref{eq:EBal}) is commonly called the \textit{energy principle} in literature \cite{Hill2015,Hill2016,Peter2018}.

By imposing $\xi=(q_0(s;p),\dot{q}_0(s;p))+O(\varepsilon)$ and $lT=m\tau+O(\varepsilon)$ in the energy balance equation, one can easily verify that the Melnikov function of Eq. (\ref{eq:MelFun}) is the leading-order term of the Taylor expansion of Eq. (\ref{eq:EBal}), i.e.,
\begin{equation}
E_b(\xi,\varepsilon)_{[0,lT]}=\varepsilon M^{m:l}(s)+O(\varepsilon^2) ,
\end{equation}
\begin{equation*}
M^{m:l}(s)=\int_0^{m\tau}\big\langle\, \dot{q}_0(t+s;p)\, , 
\, Q(q_0(t+s;p),\dot{q}_0(t+s;p),t;\tau m/l,0)\big\rangle dt ,
\end{equation*}
where we have used Eq. (\ref{eq:FGDef}) and the relation
\begin{equation}
DH(x)=DH(q,\dot{q})=\begin{pmatrix}\, DV(q)+D_q E(q,\dot{q}),& M(q)\dot{q} \,\end{pmatrix} .
\end{equation}
Before exploring the implications of this peculiar form of the Melnikov function in specific cases, it is useful to recall the definitions of \textit{subharmonic} and \textit{superharmonic} resonances \cite{NayfehM2007} in terms of $l$ and $m$. These integers define the relation between the minimal period of the orbit and that of the perturbation. A subharmonic resonance occurs when the forcing frequency is a multiple of the orbit frequency, i.e., $l\neq 1$ and $m=1$. The converse holds for a superharmonic resonance, for which we have $l=1$ and $m\neq1$. The attribute \textit{ultrasubharmonic} \cite{GH1983} indicates higher-order resonances, when both $m$ and $l$ are different from 1. 

\section{Monoharmonic forcing with arbitrary dissipation}
\label{sec:S4}
Due to their importance in the structural vibrations context, we now consider perturbations $Q$ in Eq. (\ref{eq:MechSys}) whose leading-order term in $\varepsilon$ is of the form
\begin{equation}
\label{eq:QPFAD}
\begin{array}{lcr}
Q(q,\dot{q},t;T,e,0)=e f_e \cos(\Omega t)-C(q,\dot{q}), & \,\,\,\,&  \Omega=2\pi/T ,
\end{array}
\end{equation}
where $e\in\mathbb{R}$ is a forcing amplitude parameter, $f_e \in\mathbb{R}^N$ is a constant vector of unit norm and $C(q,\dot{q})$ is a smooth, dissipative vector field. The actual forcing amplitude and the dissipative vector field are both rescaled by the value of the perturbation parameter $\varepsilon$. 

First, we discuss the possible bifurcations that single orbits can experience in such systems when perturbed into forced-damped frequency responses, then we discuss the fate of periodic orbit families. Finally, we also illustrate the implications of the Melnikov method for the phase-lag quadrature criterion used in experimental vibration analysis.

\subsection{Bifurcations from single orbits}
\label{sec:S4_1}

We consider bifurcations from the conservative orbit $\mathcal{Z}$ with $p\in\mathcal{Z}$ and seek to perform continuation with the parameter $e$. In this case, the Melnikov function takes the form
\begin{equation}
\label{eq:MPFAD}
\begin{array}{rl}
M^{m:l}(s,e)=&\displaystyle  e\int_0^{m\tau}\langle \dot{q}_0(t+s;p) ,f_e\rangle \cos\left( \frac{2l\pi}{m\tau} t\right)dt+ \\ \\ &\displaystyle-\int_0^{m\tau}\langle \dot{q}_0(t+s;p) ,C(q_0(t+s;p),\dot{q}_0(t+s;p)) \rangle dt\\ \\
=&\displaystyle w^{m:l}(s,e)-m R,
\end{array}
\end{equation}
where we have introduced the \textit{resistance}
\begin{equation}
\label{eq:resdef}
R=\int_0^{\tau}\langle \dot{q}_0(t;p) ,C(q_0(t;p),\dot{q}_0(t;p)) \rangle dt,
\end{equation}
measuring the dissipated energy along one period $\tau$ of $\mathcal{Z}$. This function is independent of $s$ since $C$ does not explicitly depend on time and the factor $m$ arises in (\ref{eq:MPFAD}) because (\ref{eq:resdef}) is $\tau$-periodic. In contrast,
\begin{equation}
\label{eq:workdef}
w^{m:l}(s,e)=e\int_0^{m\tau}\langle \dot{q}_0(t+s;p) ,f_e\rangle \cos\left( \frac{2l\pi}{m\tau} t\right)dt
\end{equation}
is the work done by the force along $m$ periods of the conservative solution.

To simplify Eq. (\ref{eq:MPFAD}) further, we express the conservative periodic solution $\mathcal{Z}$ using the Fourier series
\begin{equation}
\label{eq:Fseries}
\begin{array}{lcr}
\displaystyle q_0(t;p)=\frac{a_0}{2}+\sum_{k=1}^\infty a_k \cos\left( k\omega t\right)+b_k \sin\left( k\omega t\right) , &\,\,\,\, & \displaystyle  \omega=\frac{2\pi}{\tau} ,
\end{array}
\end{equation}
where $a_k,\,b_k \in\mathbb{R}^N$ are the Fourier coefficients of the displacement coordinates. By inserting this expansion in Eq. (\ref{eq:MPFAD}), we obtain for $w^{m:l}(s,e)$ the expression
\begin{equation}
\label{eq:WPFAD}
w^{m:l}(s,e)=\begin{cases} 0 & \mathrm{if\,\,} m\neq 1 \\
\displaystyle  W^{1:l}(e)\cos\left( l\omega s -\alpha_{l,e}\right) & \mathrm{if\,\,} m= 1
\end{cases},
\end{equation}
where
\begin{equation}
\begin{array}{lcccr}
W^{1:l}(e)=eA_{l,e} , &\,\,\,\, & A_{l,e}=l\pi\sqrt{\displaystyle\left\langle a_l, f_e\right\rangle^2+\left\langle b_l ,f_e\right\rangle^2}, &\,\,\,\, & \displaystyle \alpha_{l,e}=\arctan \left(\frac{\langle a_l ,f_e\rangle}{\langle b_l ,f_e\rangle}\right).
\end{array}
\end{equation}
We provide the details of these derivations in Appendix \ref{app:A2}. The quantity $W^{1:l}(e)$ measures the maximum work done by the forcing along one cycle of the conservative periodic orbit. This work depends linearly on the forcing amplitude parameter $e$. Equation (\ref{eq:WPFAD}) implies that superharmonics or ultrasubharmonics cannot occur for the considered perturbation, which is consistent with literature observations. As a consequence, we have the following proposition characterising primary and subharmonic resonances, where the relation between the forcing frequency $\Omega$ and the conservative orbit frequency $\omega$ reads $\Omega = l\omega +O(\varepsilon)$.
\begin{proposition}
\label{prop:MPFADclass}
The Melnikov function for the perturbation in Eq. (\ref{eq:QPFAD}) takes the specific form
\begin{equation}
\label{eq:MPFADp}
M^{1:l}(s,e)=W^{1:l}(e)\cos\left( l\omega s -\alpha_{l,e}\right)-R .
\end{equation}
Assuming $M^{1:l}(s,e_0) \not\equiv 0$ for some $e_0\neq 0$, the following bifurcations of the conservative periodic orbit $\mathcal{Z}$ are possible for small $\varepsilon >0$:
\begin{enumerate}[label=\textit{(\roman*)}]
\item if $|W^{1:l}(e_0)|<|R|$, the conservative solution $\mathcal{Z}$ does not smoothly persist;
\item if $|W^{1:l}(e_0)|>|R|$, two periodic orbits bifurcate from $\mathcal{Z}$;
\item if $|W^{1:l}(e_0)|=|R|>0$, there exist a forcing amplitude parameter $e_{sn}=e_0+O(\varepsilon)$ for which a unique periodic orbit emanates from $\mathcal{Z}$.
\end{enumerate}
\end{proposition}
\begin{proof}
Since $M^{1:l}(s,e_0)$ remains bounded away from zero for $|W^{1:l}(e_0)|<|R|$, statement \textit{(i)} follows from Theorem \ref{thm:stpert}. When $|W^{1:l}(e_0)|>|R|$, $M^{1:l}(s,e_0)$ features $2l$ simple zeros for $s\in[0,\tau)$ for which Theorem \ref{thm:stpert} applies again. Considering that the forcing signal passes $l$ times the zero phase in $[0,\tau)$, $\,l$ of these zeros correspond to a single perturbed orbit so that two periodic solutions bifurcate from $\mathcal{Z}$, proving statement \textit{(ii)}. Finally, we will argue that statement \textit{(iii)} is a direct consequence of Theorem \ref{thm:ftpert}. First, note that the Melnikov function (\ref{eq:MPFADp}) features $l$ quadratic zeros in $s$ as defined in Eq. (\ref{eq:qtanzero}), corresponding to the $l$ maxima or minima of $\cos\left( l\omega s -\alpha_{l,e}\right)$ for $s\in[0,\tau)$, depending on the signs of $W^{1:l}(e)$ and $R$. Considering a location $s_{qz}$ among these quadratic zeros, we obtain
\begin{equation}
\label{eq:DeMPFAD}
|D_e M^{1:l}(s_{qz},e_0)|=|D_e W^{1:l}(e_0)|=A_{l,e}> 0
\end{equation}
by the assumption $|W^{1:l}(e_0)|=|R|>0$. Thus, these quadratic zeros are nondegenerate in $\kappa$ and, since $l$ of them correspond again to a single orbit, we conclude that a saddle-node bifurcation occurs from Theorem \ref{thm:ftpert}. More precisely, there exists a value $e_{sn}=e_0+O(\varepsilon)$ for which a periodic orbit $O(\varepsilon)$-close to $\mathcal{Z}$ corresponds to a fold point for continuations in $e$. \\ 
Due to the specific form of the function in Eq. (\ref{eq:MPFADp}), no further degeneracies in $s$ are possible (e.g. cubic zeros) for $M^{1:l}(s,e)$ so that the cases \textit{(i-iii)} are the only possible bifurcations. 
\end{proof}
From Proposition \ref{prop:MPFADclass}, we can derive necessary conditions for the persistence of a periodic orbit under forcing and damping. Either for case \textit{(ii)} and \textit{(iii)}, $W^{1:l}(e)$ must be nonzero, i.e., $e\neq0$ and $A_{l,e}>0$. The latter quantity is zero if the $l$-th harmonic is not present in Eq. (\ref{eq:Fseries}) or if $f_e$ is orthogonal to both its Fourier vectors. In the non-generic case of $M^{1:l}(s,e_0)\equiv 0$, the Melnikov function does not give any information on the persistence problem.

\subsection{Bifurcations from normal families}
\label{sec:S4_2}
We now investigate possible bifurcations that a conservative, 1-normal family $\mathcal{P}$ of periodic orbits may exhibit when perturbed with Eq. (\ref{eq:QPFAD}) into frequency responses at fixed $e$. Specifically, we study phenomena that occur with respect to the forcing frequency $\Omega$ and an amplitude measure $a$ of interest.  

We assume that either $\omega$ or $a$ can be locally used as the family parameter $\lambda$ for $\mathcal{P}$ and we denote $\mathcal{B}$ the conservative backbone curve in the plane $(l\omega,a)$. We then introduce the following definition.
\begin{definition}
\label{def:Ridge}
A \textit{ridge} $\,\mathcal{R}_l$ is the curve in the plane $(e,\lambda)$ identifying the forcing amplitudes and the orbits of $\mathcal{P}$ at which quadratic zeros of $M^{1:l}$ in $s$ occur.
\end{definition}
The significance of ridges for frequency responses is clarified by the following proposition.
\begin{proposition}
\label{prop:ridgeprop}
Assume that $eR(\lambda_0)>0$ and $A_{l,e}(\lambda_0)>0$ hold for the periodic orbit $\mathcal{Z}$ identified by $\lambda_0$. Then, the explicit local definition of $\,\mathcal{R}_l$ becomes $e=\Gamma_l(\lambda)$, where
\begin{equation}
\Gamma_l(\lambda)=R(\lambda)/A_{l,e}(\lambda) . 
\end{equation}
If $D\Gamma_l(\lambda_0)> 0$ (resp. $<0$), then the forced-damped response for $e_0=\Gamma_l(\lambda_0)$ shows a maximal (resp. minimal) response with respect to $\lambda$ $O(\varepsilon)$-close to $\mathcal{B}$ at $\mathcal{Z}$. If $D\Gamma_l(\lambda_0)=0$ and $D^2\Gamma_l(\lambda_0)>0$ (resp. $<0$), then the forced-damped response for $e_0=\Gamma_l(\lambda_0)$ includes an isola birth (resp. simple bifurcation) in $\lambda$ which is $O(\varepsilon)$-close to $\mathcal{B}$ at $\mathcal{Z}$.
\end{proposition}
\begin{proof}
We rewrite the Melnikov function as
\begin{equation}
M^{1:l}(s,e,\lambda)=A_{l,e}(\lambda)\Big(e\cos\big(l\omega(\lambda) s -\alpha_{l,e}(\lambda)\big)-\Gamma_l(\lambda)\Big) ,
\end{equation}
which features $l$ quadratic zeros in $s$ for $e=\Gamma_l(\lambda)$. When $D\Gamma_l(\lambda_0)\neq 0$, Theorem \ref{thm:ftpert} identifies a saddle-node bifurcation because
\begin{equation}
\label{eq:dml1}
D_\lambda M^{1:l}\big(s_{qz}(\lambda_0),\Gamma_l(\lambda_0),\lambda_0\big)=-A_{l,e}(\lambda_0)D\Gamma_l(\lambda_0)\neq 0 ,
\end{equation}
at any of the $l$ locations $s_{qz}(\lambda_0)$ of quadratic zeros of $M^{1:l}$ in $s$. As already discussed in Proposition \ref{prop:MPFADclass}, there exists a unique periodic orbit $O(\varepsilon)$-close to $\mathcal{Z}$, corresponding to a fold for continuations in $\lambda$. If $D\Gamma_l(\lambda_0)>0$, we can choose a small positive $\epsilon$ defining a $\lambda_1=\lambda_0-\epsilon$ for which
\begin{equation}
W^{1:l}(e_0,\lambda_1)=e_0A_{l,e}(\lambda_1)=\Gamma_l(\lambda_0)A_{l,e}(\lambda_1)>\Gamma_l(\lambda_1)A_{l,e}(\lambda_1)=R(\lambda_1) ,
\end{equation}
so that, according to Proposition \ref{prop:MPFADclass}, two periodic orbits bifurcate at $e_0=\Gamma_l(\lambda_0)$ from the orbit of $\mathcal{P}$ described by $\lambda_1$. For $\lambda_2=\lambda_0+\epsilon$, we can similarly conclude that no orbit persists smoothly. Thus, the periodic orbit at the fold in $\lambda$ represents a maximal response. An analogous reasoning holds for the minimal response case arising for $D\Gamma_l(\lambda_0)<0$.

The last statement of Proposition \ref{prop:ridgeprop} holds again, based on Theorem \ref{thm:ftpert}, since we have that $D_\lambda M^{1:l}\big(s_{qz}(\lambda_0),\Gamma_l(\lambda_0),\lambda_0\big)=0$ from Eq. (\ref{eq:dml1}) and
\begin{equation}
\mathrm{det}\big( D^2_{s,\lambda} M^{1:l}\big(s_{qz}(\lambda_0),\Gamma_l(\lambda_0),\lambda_0\big)\big)=e\big(A_{l,e}(\lambda_0)\omega(\lambda_0)\big)^2D^2\Gamma_l(\lambda_0)\neq 0 .
\end{equation}
\end{proof} 
Ridges, as introduced in Definition \ref{def:Ridge}, are effective tools for the analysis of forced-damped responses in the vicinity of backbone curves as they can track fold bifurcations and generations of isolated responses. These phenomena are the most generic bifurcations for the perturbation type in Eq. (\ref{eq:QPFAD}). Ridge points may be used to detect further singular bifurcation behaviours under additional degeneracy conditions on $\lambda$ \cite{GolubitskySchaeffer1985}. 

\subsection{The phase-lag quadrature criterion}
\label{sec:S4_3}
We now discuss the relevance of the phase of the Melnikov function and the next proposition illustrates an important result in this regard.
\begin{proposition}
\label{prop:quadzerolag}
Consider a perturbed periodic orbit $q_{qz}(t;\xi,T,\varepsilon)$ corresponding to a quadratic zero of the Melnikov function (\ref{eq:MPFADp}) related to the conservative limit $\mathcal{Z}$. Then, the $l$-th harmonic of the function $\langle q_{qz}(t;\xi,T,\varepsilon) ,f_e\rangle$ has a phase lag (resp. lead) of $\,90^\circ+O(\varepsilon)$ with respect to the forcing signal if $eR>0$ (resp. $eR<0$).
\end{proposition}
\begin{proof} 
To determine the phase lag, we consider, without loss of generality, the phase condition for $\mathcal{Z}$
\begin{equation}
\label{eq:paraqpc}
\begin{array}{lcr} 
\langle a_l ,f_e\rangle>0, &\,\,\,\,&
\langle b_l ,f_e\rangle=0 ,
\end{array}
\end{equation}
under which the $l$-th term in the Fourier series of the function $\langle q_0(t;p) ,f_e\rangle$ is equal to $\langle a_l ,f_e\rangle \cos(l\omega t)$, having the same phase of the forcing. In that case, the Melnikov function becomes
\begin{equation}
\label{eq:melnphaselag}
M^{1:l}(s,e)=W^{1:l}(e)\cos\left( l\omega s +3\pi/2\right)-R=-W^{1:l}(e)\sin\left( l\omega s \right)-R .
\end{equation}
Eq. (\ref{eq:melnphaselag}) shows that the $l$ quadratic zeros of the Melnikov function occur for $|W^{1:l}(e)|=|R|$ and $l\omega s_{qz}=-\mathrm{sign}(eR)\pi/2+2k\pi$ with $k=0,1,...\,l-1$. Thus, we obtain
\begin{equation}
\langle q_{qz}(t;\xi,T,\varepsilon) ,f_e\rangle=\langle q_0(t+s_{qz};p) ,f_e\rangle+O(\varepsilon) ,
\end{equation}
whose $l$-th harmonic is equal to $\langle a_l ,f_e\rangle \cos(l\omega t -\mathrm{sign}(eR)\pi/2)+O(\varepsilon)$, independent of $k$.
\end{proof}
In numerical or experimental continuation, one can track the relation between the forcing amplitude parameter $e$ and either the amplitude $a$ or the forcing frequency $\Omega$ under the phase criterion of Proposition \ref{prop:quadzerolag}. The resulting curve of points is an $O(\varepsilon)$-approximation of the ridge curve $\,\mathcal{R}_l$ whose interpretation is available in Proposition \ref{prop:ridgeprop}.

Proposition \ref{prop:quadzerolag} relaxes some restrictions of the phase-lag quadrature criterion derived in \cite{Peeters2011a}. Indeed, Eq. (\ref{eq:Fseries}) allows for arbitrary periodic motion, not just synchronous ones along which all displacement coordinates reach their maxima at the same time. Moreover, our criterion is not limited to velocity-dependent, odd damping, but it admits arbitrary, smooth dissipations. For this general case, we proved that the phase-lag must be measured in \textit{co-location}, i.e., when the output (displacement response) is observed at the same location where the input (force) excites the system.

\section{Examples}
\label{sec:S5}
In this section, we study a conservative multi-degree of freedom system subject to non-conservative perturbations in the form of Eq. (\ref{eq:QPFAD}). First, we consider frequency responses with monoharmonic forcing and linear damping. Then, we introduce nonlinear damping to investigate the presence of isolas. In both cases, we show how the Melnikov analysis can predict forced-damped response bifurcations under a $1:1$ resonance between the forcing and periodic orbits of the conservative limit.

We analyse a system composed of six masses $m_i$ with $i=1,\,2,\, ... \, 6$ that are connected by seven nonlinear massless elements, as shown in Fig. \ref{fig:S5_0}. All masses are assumed unitary and the external forcing acts on the first degree of freedom only. The seven nonlinear elements exert a force depending on the elongation $\Delta l$ and its speed $\dot{\Delta l}$, modelled as
\begin{equation}
\label{eq:Feldamp}
\begin{array}{rl}
F_{i}(\Delta l,\dot{\Delta l},\varepsilon)=F_{i,el}(\Delta l)+\varepsilon F_{i,nc}(\dot{\Delta l})=& k_{i,1}\Delta l+k_{i,3}\Delta l^3+k_{i,5}\Delta l^5+\\ &+ \varepsilon (\alpha k_{i,1}\dot{\Delta l}+\beta k_{i,3}\dot{\Delta l}^3+\gamma k_{i,5}\dot{\Delta l}^5)
\end{array}
\end{equation}
\begin{figure}[t!]
\centering
\includegraphics[width=1\textwidth]{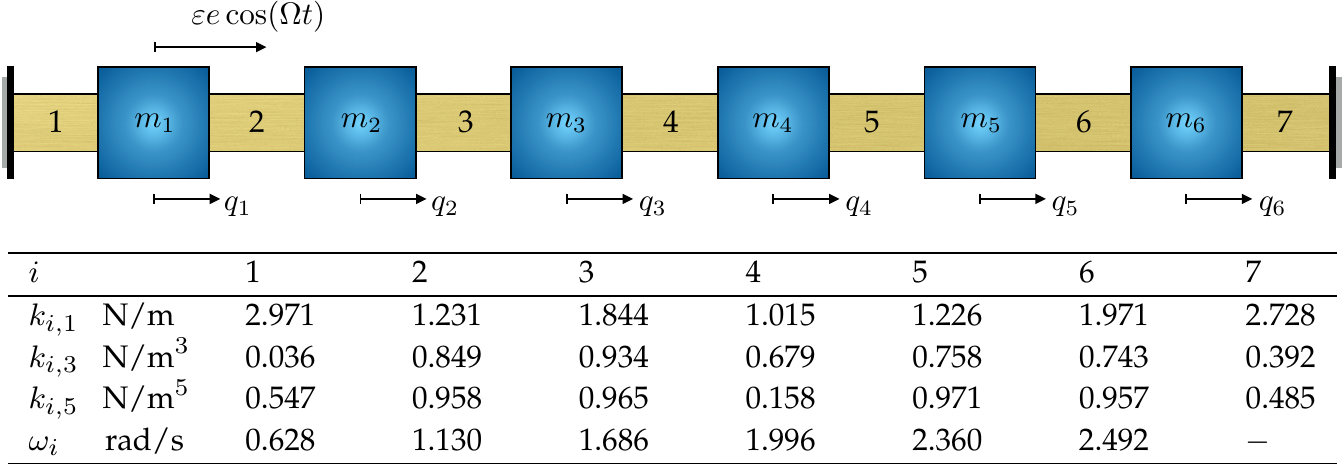}
\caption{Illustration of the mechanical system in (\ref{eq:mechsysex}) and table containing elastic coefficients $k_{i,j}$ of the constitutive law in (\ref{eq:Feldamp}) for the nonlinear elements and natural frequencies $\omega_i$ of the system linearised at the origin.}
\label{fig:S5_0}{}
\end{figure}
for $i=1,\,2,\, ... \, 7$. The coefficients $k_{i,1}$, $k_{i,3}$ and $k_{i,5}$ are reported in the table in Fig. \ref{fig:S5_0}, while the values of $\alpha$, $\beta$, $\gamma$ and $\varepsilon$ will vary from case by case below. The equations of motion read
\begin{equation}
\label{eq:mechsysex}
\begin{cases}
\ddot{q}_1+F_{1}(q_1,\dot{q}_1,\varepsilon)+F_{2}(q_1-q_2,\dot{q}_1-\dot{q}_2,\varepsilon)=\varepsilon e \cos(\Omega t) , \\ 
\,\,\,\,\,\,\,\,  \vdots\\
\ddot{q}_i+F_{i}(q_i-q_{i-1},\dot{q}_i-\dot{q}_{i-1},\varepsilon)+F_{i+1}(q_i-q_{i+1},\dot{q}_i-\dot{q}_{i+1},\varepsilon)=0 , & \mathrm{for}\,\,i=2,\,3,\, ... \, 5 , \\ 
\,\,\,\,\,\,\,\,\vdots\\
\ddot{q}_6+F_{6}(q_6-q_5,\dot{q}_6-\dot{q}_5,\varepsilon)+F_{7}(q_6,\dot{q}_6,\varepsilon)=0 .
\end{cases}
\end{equation}
To compute conservative periodic orbits and frequency responses for system (\ref{eq:mechsysex}), we use the MATLAB-based numerical continuation package \textsc{coco} \cite{Dankowicz2013}. We specifically exploit its periodic orbit toolbox that solves the continuation problem via collocation. In this method, solutions to the governing ordinary differential equations are approximated by piecewise polynomial functions and continuation is performed using a refined pseudo-arclength algorithm.

First, we focus on the study of the conservative limit ($\varepsilon=0$), in which the nonlinear elements are springs with convex potentials and the origin is an equilibrium whose eigenfrequencies $\omega_i$ are reported in the table of Fig. \ref{fig:S5_0}. Since no resonance arises among these frequencies, the system features six families of periodic orbits emanating from the origin by the Lyapunov subcenter manifold theorem \cite{Lyapunov1992}. Using numerical continuation starting from small-amplitude linearised periodic motions, we compute the conservative backbone curve for each mode, shown in Figure \ref{fig:S5_I}(a). We plot these curves using the normalised frequency $\bar{\omega}=\omega/\omega_1$ and the $L^2$ norm $||x||_{L^2,[0,T]}$ of the conservative periodic orbits. We consider the latter norm as the amplitude measure $a$. With the exception of the first periodic orbit family, the monodromy matrix of the periodic orbits has two Floquet multipliers equal to $+1$, whose geometric multiplicity is 1 in the selected frequency-amplitude range. Therefore, these five orbit families are $1$-normal, precisely belonging to case \textit{(a)} of Definition \ref{def:mNPO}, and showing a hardening trend ($Da$, $D\omega>0$). The first family also shows normality with hardening behaviour up to the magenta point, where branching phenomena takes place and a further family originates from the continuation of the first linearised mode. As $1$-normality does not hold in the vicinity of the branch point, depicted in magenta in Fig. \ref{fig:S5_I}(a), we restrict our analysis of the first family to amplitudes below the branch point amplitude.
\begin{figure}[t!]
\centering
\includegraphics[width=1\textwidth]{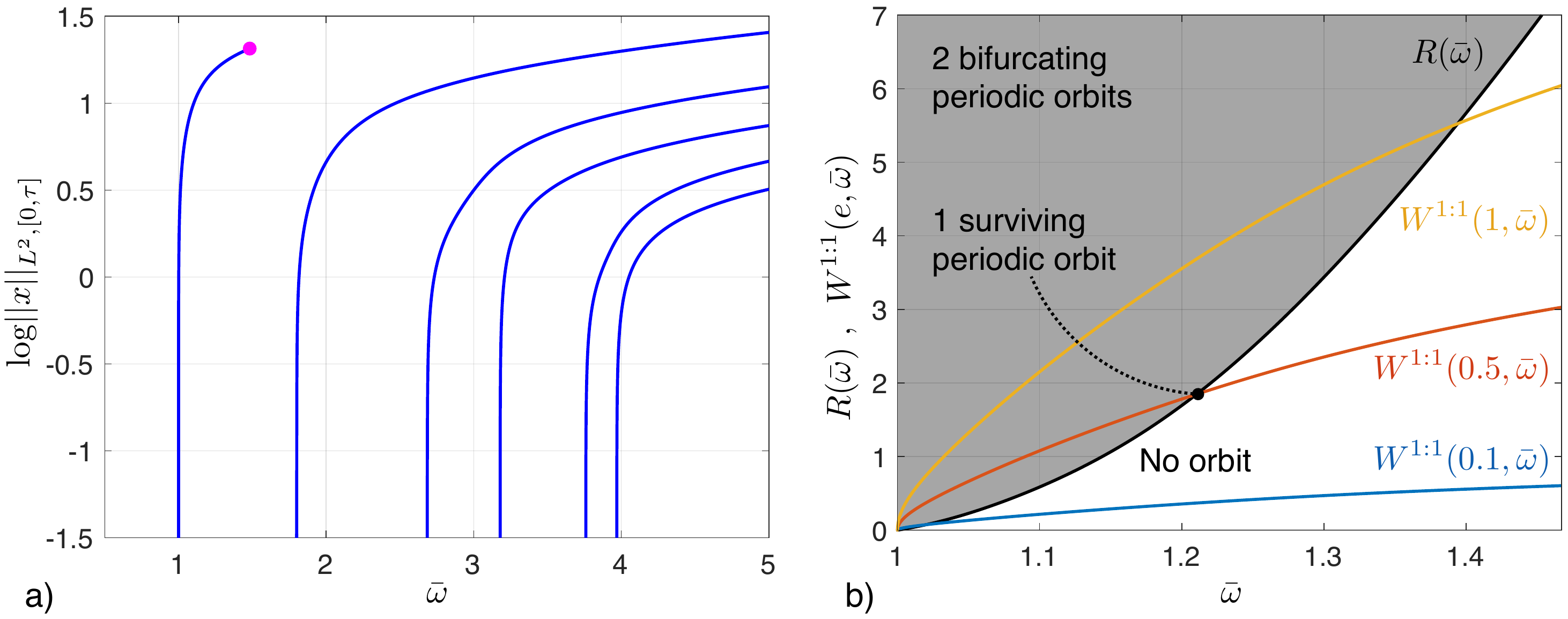}
\caption{(a) Conservative backbone curves of the unperturbed system and (b) Melnikov analysis for the first mode of the system with linear damping $\alpha=0.04$: the black solid line is the resistance $R(\bar{\omega})$; coloured lines show the amplitude of the active work $W^{1:1}_{a}(e,\bar{\omega})$ for different forcing amplitude values.}
\label{fig:S5_I}
\end{figure}

\subsection{Resonant external forcing with linear damping}
In this first example, we take the damping to be linear with $\alpha=0.04$ and $\beta=\gamma=0$ in Eq. (\ref{eq:Feldamp}). We focus on the orbits surviving from primary resonance conditions, where $m=l=1$, and we perform the analysis of $M^{1:1}$ for each mode of the system as described in section \ref{sec:S4}.

Figure \ref{fig:S5_I}(b) shows the work done by non-conservative contributions along the first modal family of conservative orbits parametrised with the non dimensional frequency $\bar{\omega}$. The black solid line is the resistance $R(\bar{\omega})$, while coloured lines represent $W^{1:1}(e,\bar{\omega})$ for three forcing amplitudes. According to Proposition \ref{prop:MPFADclass}, we find that two orbits bifurcate from the conservative one when the lines illustrating $W^{1:1}(e,\bar{\omega})$ lay in the grey zone of this plot, i.e., when $W^{1:1}(e,\bar{\omega})>R(\bar{\omega})$. No orbit bifurcates in the white area and unique solutions appear at intersection points between coloured lines and the black one. We also note that $A_{1,e}$ is never zero, except when $\bar{\omega}=1$. Similar trends and considerations hold for the other modes, except for the last one. For that mode, the active work contribution of the forcing is very small compared to the dissipative terms: the forcing is nearly orthogonal to the mode shape. Thus, no orbits arise from the conservative limit for the forcing amplitude ranges under investigation.

Figure \ref{fig:S5_II}(c) shows the curves $\Gamma_1(\bar{\omega})$ for the first five modes of the system using different colours; all of them show a strictly increasing monotonic trend. Thus, according to Proposition \ref{prop:ridgeprop}, ridge orbits are $O(\varepsilon)$ approximations for maximal responses in $\omega$ and $a$, since all the conservative backbone curves can be parametrised with both quantities. By selecting a forcing value in Fig. \ref{fig:S5_II}(c), we can predict the frequencies and the amplitudes of the maximal responses in the forced-damped setting. Moreover, since the damping is linear and proportional, ridges are defined as
\begin{equation}
\label{eq:lindamprid}
e= \frac{R(\bar{\omega})}{A_{1,e}(\bar{\omega})}=\alpha \frac{1}{A_{1,e}(\bar{\omega})}\int_0^{\tau(\bar{\omega})} \langle \dot{q}_0(t;p(\bar{\omega}) ) , K \dot{q}_0(t;p(\bar{\omega}) ) \rangle dt ,
\end{equation}
where we denoted with $K$ the stiffness matrix of system (\ref{eq:mechsysex}) and expressed initial conditions $p$, periods $\tau$ and coefficients $A_{1,e}$ as functions of the non-dimensional frequency $\bar{\omega}$. From Eq. (\ref{eq:lindamprid}), we obtain that the location of maximal frequency responses close to backbone curves is determined by the ratio between the forcing amplitude parameter $e$ and the damping term $\alpha$, with $O(\varepsilon)$ accuracy.
\begin{figure}[t!]
	\centering
	\includegraphics[width=1\textwidth]{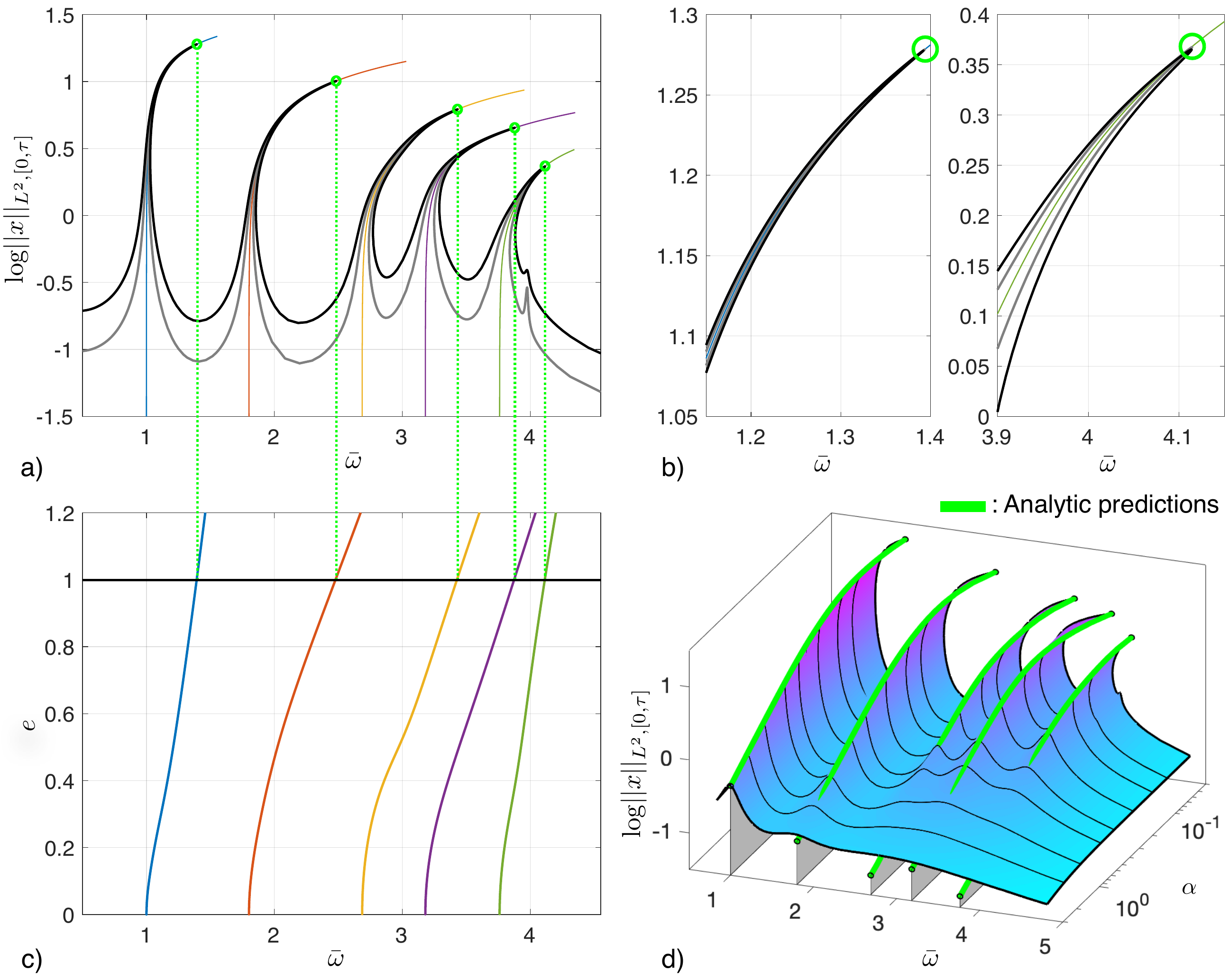}
	\caption{Plots (a) and (b) shows frequency responses with $\alpha=0.04$ and $e=1$ for $\varepsilon=0.05$, grey line, and $\varepsilon=0.1$, black line. The second plot zooms near the first and fifth peaks of the first plot. The five relevant conservative periodic orbit families are highlighted with coloured lines. Plot (c) shows the ridges $\mathcal{R}_1$ for each mode in different colours and the black line represents the forcing amplitude parameter of the frequency response in (a), so that, carrying over this intersection frequencies with green dotted lines, we obtain an analytic approximation for turning points. These approximations are described by green circles in (c). Plot (d) shows the frequency response surface with $e=1$ and $\varepsilon=0.1$, varying the proportional damping term $\alpha$ completed with conservative families in grey surfaces and analytic predictions for maxima in green.}
	\label{fig:S5_II}
\end{figure}

These theoretical findings are confirmed by the direct numerical computation of frequency responses presented in Fig. \ref{fig:S5_II}. To obtain them, we continue in frequency an initial guess acquired through numerical integration for a forcing frequency away from resonance with any of the linearised natural frequencies. The existence of this orbit is guaranteed by the asymptotic stability of the origin when $\varepsilon>0$ and $e=0$. Plot (a) in Fig. \ref{fig:S5_II} shows two frequency sweeps for $e=1$ and for $\varepsilon = 0.05$ (grey line), $0.1$ (black line), while this plot is zoomed in (b) around the first and fifth peaks. The sixth mode shows some tiny responses with the rightmost peaks in these two frequency sweeps, more evident for the case $\varepsilon = 0.05$ where the physical damping is lower. Figures \ref{fig:S5_II}(a) and \ref{fig:S5_II}(b) are completed with our analytic predictions in green for the maxima. By imposing the forcing parameter in the ridges of Fig. \ref{fig:S5_II}(c), we obtain the frequencies of each mode around which maximal response occur that are validated when carried over with green dotted lines in Fig. \ref{fig:S5_II}(a). Moreover, Fig. \ref{fig:S5_II}(d) shows the frequency response surface keeping $e=1$ and varying the damping value\footnote{For purposes of better illustration, we decided to sweep with the damping parameter $\alpha$ instead of the forcing amplitude one.} for two orders of magnitude. Green curves show analytic predictions for the maxima that closely approximate the peaks of this surface.
\begin{figure}[t]
\centering
\includegraphics[width=1\textwidth]{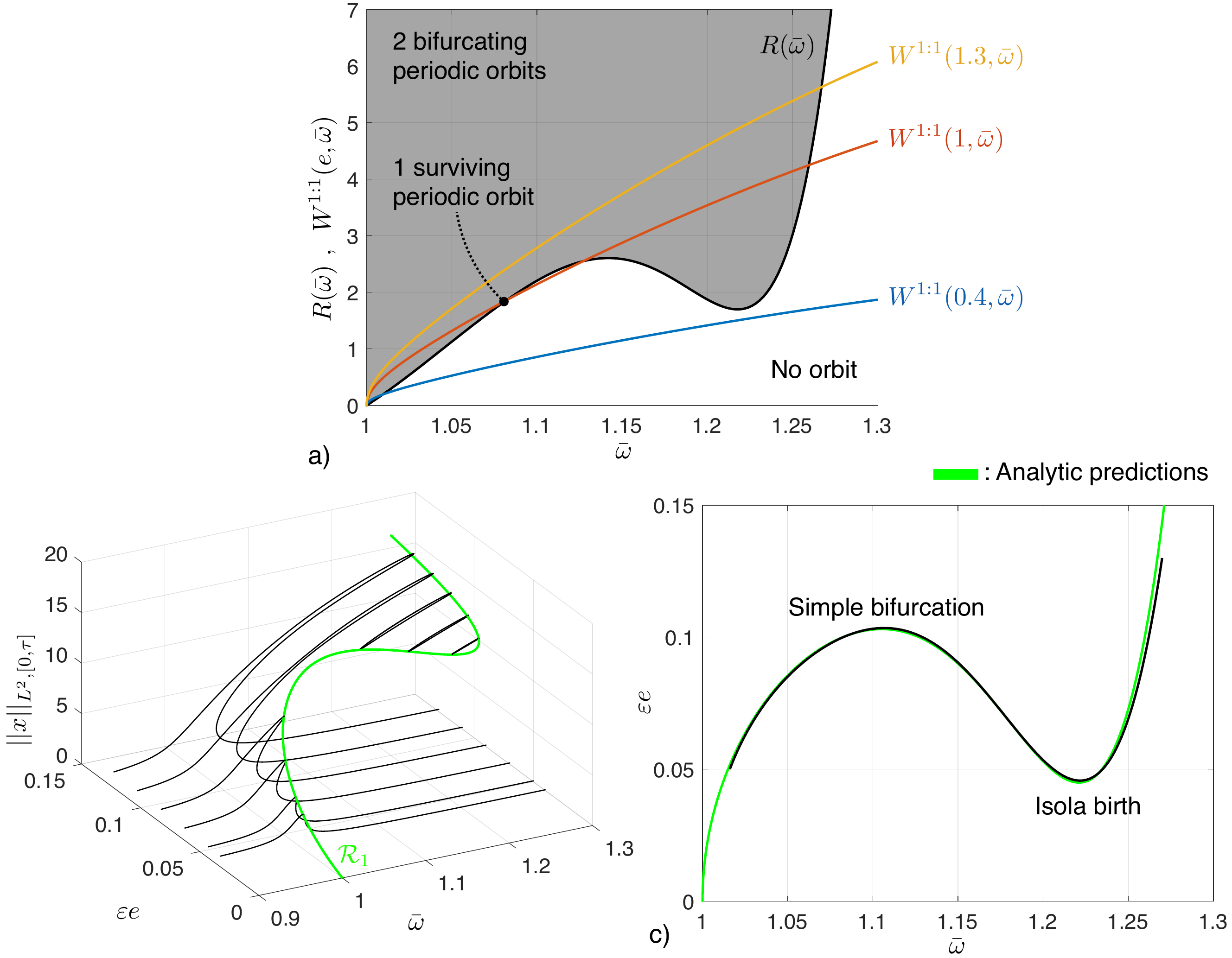}
\caption{Plot (a) shows the Melnikov analysis for $\alpha=0.2481$, $\beta=-1.085$ and $\gamma=0.8314$ regarding the first mode. Plot (b) shows frequency responses varying the forcing amplitude parameter and fixing $\varepsilon=0.1$ with the ridge curve $\mathcal{R}_1$. The latter is compared in plot (c) with the relation between the $e$ and $\bar{\omega}$ obtained through the numerical continuation of saddle-nodes orbits occurring close to the maximal point of the frequency response.}
\label{fig:S5_III}
\end{figure}

\subsection{Resonant external forcing with nonlinear damping}
We now repeat the analysis of the previous section including also the nonlinear damping characteristic of the connecting elements. In order to break the monotonic trend of the resistance in the linear damping case, cf. Fig. \ref{fig:S5_I}(b), we select $\alpha=0.2481$, $\beta=-1.085$ and $\gamma=0.8314$. We also restrict our attention solely to the first mode of the system.

The Melnikov analysis is reported in Figure \ref{fig:S5_III}(a), which outlines a behaviour change for increasing forcing amplitudes. Indeed, an isola birth occurs at $e\approx 0.4$ as was also displayed in Fig. \ref{fig:S3_IM3}. The branch persists up to connecting with the main branch for $e\approx 1$ through a simple bifurcation.

These predictions are confirmed by the numerical simulations shown in Figures \ref{fig:S5_III}(b) and \ref{fig:S5_III}(c). The former illustrates several frequency responses for different physical forcing amplitudes, with $\varepsilon=0.1$. The green line is the ridge $\mathcal{R}_1$, also plotted in Fig. \ref{fig:S5_III}(c), and the two singular bifurcations show up at its folds, as explained in Proposition \ref{prop:ridgeprop}. From a computational perspective, main branches of the frequency response are computed with the same strategy of the previous section. For isolated branches, we obtain initial guesses from a numerical continuation in $e$ of saddle-node periodic orbits\footnote{This functionality is directly available in the periodic orbit toolbox of \textsc{coco} \cite{Dankowicz2013} through the constructor \texttt{ode\_SN2SN}.} that started near the maximal response of the frequency sweep at $e=1.3$. We also plot the relation between frequency and forcing amplitude in this latter numerical continuation with the black line of Fig. \ref{fig:S5_III}(c). This curve is $O(\varepsilon)$-close to the ridge (in green), which was obtained solely from the knowledge of the conservative limit. We remark that a 18-core workstation with 2.3 GHz processors required 18 minutes and 15 seconds to compute the black curve, while the green curve took 1 minute and 45 seconds to compute.

\section{Conclusion}
We have developed an analytic criterion that relates conservative backbone curves to forced-damped frequency responses in multi-degree-of-freedom mechanical system with small external forcing and damping. Our procedure uses a perturbation approach starting from the conservative limit to evaluate the persistence or bifurcation of periodic orbits in the forced-damped setting. We have shown that this problem can be reduced to the analysis of the zeros of a Melnikov-type function. In a general setting, we proved that, if a simple zero of the Melnikov function exists, generically two periodic orbits bifurcate the conservative limit. We also characterised quadratic zeros and eventual singular bifurcations that may arise. Our results assume the forcing to be periodic and small, but otherwise allow for arbitrary types of damping and forcing. In addition, our analysis yields analytic criteria for the creation of subharmonics, superharmonics and ultrasubharmonics arising from small forcing and damping.

When applied specifically to mechanical systems, the Melnikov function turns out to be the leading-order term in the equation expressing energy balance over one oscillation period. In this context, we have worked out the Melnikov function in detail for the typical case of purely sinusoidal forcing combined with an arbitrary dissipation. Our method shows that either two, one or no orbits can arise from an orbit of the conservative limit. Moreover, ridge curves allow to identify forcing amplitudes and orbits of conservative backbone curves that are close to bifurcations phenomena of the frequency response. Thus, saddle-node bifurcations of frequency continuations, maximal responses and isolas can be efficiently predicted directly from the analysis of the conservative limit of the system. Our analysis also justifies the phase-lag quadrature criterion of \cite{Peeters2011a} in a general setting, without the assumptions of synchronous motion and linear damping.

We have confirmed these theoretical findings by numerical simulations. Specifically, we have considered a nonlinear mechanical system with six degrees of freedom, and implemented our Melnikov analysis on the six families of periodic orbits emanating from an equilibrium. We have verified our results both for linear and nonlinear damping. In the latter case, we successfully predicted the generation of isolated branches in the frequency response. Our six-degree-of-freedom example illustrates that the analytic tools developed here do not require the conservative limit of the mechanical system to be integrable. Indeed, one can apply the present Melnikov function approach directly to periodic orbit families obtained from numerical continuation in the conservative limit of the system.

Our present analysis is limited to well-defined conservative one-parameter families of periodic orbits subject to small damping and periodic external forcing. Perturbations of degenerate cases or resonance interactions, which can be identified by analysing the monodromy matrix of a conservative orbit, would require the analysis of a more general, multi-dimensional bifurcation function. A rigorous approach for tackling such problems can be found in \cite{ChiconeR2000}.
\vskip12pt

\noindent\textbf{Acknowledgements.} We are grateful to Harry Dankowicz and Jan Sieber for organizing the Advanced Summer School on Continuation Methods for Nonlinear Problems 2018 with the financial support from the ASME Design Engineering Division and the ASME Technical Committee on Multibody Systems and Nonlinear Dynamics. We are also thankful to Thomas Breunung and Shobhit Jain for their careful proofreading of this paper.

\begin{appendix}
\section{Proofs of the theorems in section \ref{sec:S3}}
\label{app:A1}
\subsection{Preparatory results}
We first need some technical results to set the stage for the proofs of the theorems stated in section \ref{sec:S3}. For a one-parameter family $\mathcal{P}$ of periodic orbits emanating from a $m$-normal periodic orbit $\mathcal{Z}$, the smooth map $\mathcal{T}:\mathbb{R}\rightarrow\mathbb{R}$ describes the minimal period $\mathcal{T}(\lambda)$ of each orbit. Introducing a Poincaré section $\mathcal{S}$ passing through the point of $z\in\mathcal{P}$, we can find a smooth curve $\vartheta_\mathcal{S}:\mathbb{R}\rightarrow V\cup \mathcal{P}\pitchfork\mathcal{S}$ parametrising initial conditions under $\lambda$ where $V\subset\mathbb{R}^n$ is a open neighbourhood of $z$. For more detail on these mappings, we refer the reader to \cite{Teschl2012}. We denote the tangent space of the 2-dimensional manifold $\mathcal{P}$ at the point $z$ by $\mathbb{T}_z\mathcal{P}$, to which $f(z)$ belongs due to invariance. We consider vectors as column ones and we use the superscript $^*$ to denote transposition. We refer to the column and row spaces of a matrix $A$ with the notations col$(A)$ and row$(A)$, respectively.

Next, we discuss a useful result on the properties of the monodromy matrices for normal periodic orbits. Specifically, we restate Proposition 2.1 of \cite{ChiconeR2000} for the setting of $m$-normal period orbits.
\begin{proposition}
\label{prop:fact}
Consider an $m$-normal periodic orbit $\mathcal{Z}$ of period $m\tau$ in the periodic orbit family $\mathcal{P}$. The smooth invertible matrix families $K,\,R: \mathcal{Z}\rightarrow\mathbb{R}^{n\times n}$, defined as
\begin{equation}
\begin{array}{c}
K(z)=[\,K_r(z) \,\,\,\, v(z) \,\,\,\, f(z)] , \,\,\,\,\,\,\,\, v(z)\in\mathbb{T}_z\mathcal{P}:\langle v(z), f(z)\rangle=0, \, \langle v(z), v(z)\rangle =1 , \\ \\
\mathrm{col}(K_r(z))= \mathbb{T}_z^\perp\mathcal{P}, \,\,\,\,\,\, R(z)=\begin{bmatrix} R_r(z) \\ -f^*(z) \\ DH(z) \end{bmatrix}, \,\,\,\,\,\, \mathrm{row}(R_r(z))= \mathrm{span}^\perp\{f(z)\,,\,DH(z)\},
\end{array}
\end{equation}
satisfy the identity
\begin{equation}
R(z)(\Pi^m(z)-I)K(z)=\begin{bmatrix} A_r(z) & 0 & 0 \\ w^*(z) & m\tau_v & 0 \\ 0 & 0 & 0 \end{bmatrix} ,  \,\,\,\,\,\,  w^*(z)=-f^*(z)(\Pi^m(z)-I)K_r(z) ,
\end{equation}
where $A_r(z)\in\mathbb{R}^{(n-2)\times(n-2)}$ is always invertible and the value $\tau_v\in\mathbb{R}$ describes the shear effect within $\mathcal{P}$ being zero if $\mathcal{Z}$ is a normal periodic orbit of case (b) and nonzero for case (a) of Definition \ref{def:mNPO}.
\end{proposition}
\begin{proof} 
For proving this factorisation result, we need to characterise kernel and range of the monodromy operator for $\mathcal{Z}$ based at $z$. First, by \cite{Doedel2003}, we have that
\begin{equation}
\label{eq:basKRc}
f(z)\in\mathrm{ker}(\Pi^m(z)-I) , \,\,\,\,\,\,\,\, DH(z)\in\mathrm{range}^\perp(\Pi^m(z)-I) .
\end{equation}
Without loss of generality, we introduce a Poincaré section $\mathcal{S}$ orthogonal to $f(z)$ and the value $\lambda_z$ identifies $\mathcal{Z}$ leading to $z=\vartheta_\mathcal{S}(\lambda_z)$,  $\tau=\mathcal{T}(\lambda_z)$. Consider the identity
\begin{equation}
x_0(m\mathcal{T}(\lambda);\vartheta_\mathcal{S}(\lambda))=\vartheta_\mathcal{S}(\lambda) ,
\end{equation}
whose differentiation in $\lambda$ and evaluation at $\lambda=\lambda_z$ yields
\begin{equation}
(\Pi^m(z)-I)D\vartheta_\mathcal{S}(\lambda_z)=-mD\mathcal{T}(\lambda_z)f(z) .
\end{equation}
We then have $v(z)=D\vartheta_\mathcal{S}(\lambda_z)/||D\vartheta_\mathcal{S}(\lambda_z)||$ leading to a parametrisation-independent relation and to the definition of $\tau_v$ in the form
\begin{equation}
\begin{array}{lcr}
(\Pi^m(z)-I)v(z)=-m\tau_v f(z) , & \,\,\,\,  &\displaystyle \tau_v=\frac{D\mathcal{T}(\lambda_z)}{||D\vartheta(\lambda_z)||} .
\end{array}
\end{equation}
If the orbit $\mathcal{Z}$ belongs to the case \textit{(a)} of Definition \ref{def:mNPO}, $\tau_v$ cannot be zero, otherwise the kernel of $\Pi^m(z)-I$ is two dimensional. Instead, for case \textit{(b)}, $\tau_v$ must be zero, otherwise there exists a nonzero vector $v(z)$ whose image is parallel to $f(z)$. In both cases, the column space of $K_r(z)$ always lays in the complement of the kernel of $\Pi^m(z)-I$ by construction, so it maps through $\Pi^m(z)-I$ a $n-2$ dimensional linear subspace $V_z$ such that $f(z),\,DH(z)\notin V_z$. Since the row space of the matrix $R_r(z)$ does not contain the latter vectors, the matrix $A_r(z)=R_r(z)(\Pi^m(z)-I)K_r(z)$ is invertible $\forall z\in\mathcal{Z}$. 
\end{proof}

We can now state and prove the following reduction theorem.
\begin{theorem}
\label{thm:redthm}
Perturbed solutions of Eq. (\ref{eq:DisplMapC}) in the form of Eq. (\ref{eq:orbeclosex}) are (locally) in one-to-one correspondence with the zeros of the bifurcation function
	\begin{equation}
	\label{eq:BifFun}
	B^{m:l}_L(s,\varepsilon)=M^{m:l}(s)+O(\varepsilon) ,
	\end{equation}
	where the leading-order term, defined in Eq. (\ref{eq:MelFun}), is independent from the choice of the mapping $L$ used in the last equation of system (\ref{eq:DisplMapC}).
\end{theorem}
\begin{proof} 
With the shorthand notation $z=x_0(s;p)$, we consider the following change of coordinates
\begin{equation}
\label{eq:fcoordc}
\hat{\delta} \in\mathbb{R}^{n-1} , \,\,\,\, \hat{\sigma}\in\mathbb{R} , \,\,\,\,  \begin{pmatrix}\xi \\ T  \end{pmatrix}=\hat{\varphi}(\hat{\delta},\hat{\sigma},s)=\begin{cases} z+K(z)\begin{pmatrix} \hat{\delta} \\ 0 \end{pmatrix}=z+K_{\mathbb{T}\mathcal{Z}}(z)\hat{\delta} \\ (\tau m+\hat{\sigma})/l \end{cases} ,
\end{equation}
where $K(z)$ is the matrix defined in Proposition \ref{prop:fact}. By construction, $D\hat{\varphi}(0,\hat{\sigma},s)$ is invertible for any $s\,,\,\hat{\sigma}\in\mathbb{R}$. Then, we rescale $\hat{\delta}=\varepsilon\delta$, $\hat{\sigma}=\varepsilon\sigma$ and, by calling $\eta=(\delta,\sigma,s)$, we denote $\varphi(\eta,\varepsilon)=\hat{\varphi}(\varepsilon\delta,\varepsilon\sigma,s)$. Note that $\varphi(\cdot,\varepsilon)$ is a family of diffeomorphisms for $\varepsilon$ nonzero small enough. Note also that col$(K_{\mathbb{T}\mathcal{Z}}(z))=\mathbb{T}^\perp_z\mathcal{Z}$.

By imposing this coordinate change and Taylor expanding in $\varepsilon$ Eq. (\ref{eq:DisplMapC}), we obtain $\Delta_{l,L}(\varphi(\eta,\varepsilon),\varepsilon)=\varepsilon \Delta_1(\eta,\varepsilon)$. The latter mapping is of $C^{r-1}$ class and reads
\begin{equation}
\Delta_1(\eta,\varepsilon)=
\begin{cases} D_\xi\Delta(z,m\tau/l,0)K_{\mathbb{T}\mathcal{Z}}(z) \delta +D_T\Delta(z,m\tau/l,0)\sigma +D_\varepsilon\Delta(z,m\tau/l,0) \\
D_\xi L(z,m\tau)K_{\mathbb{T}\mathcal{Z}}(z) \delta+ D_T L(z,m\tau)\sigma
\end{cases}+O(\varepsilon), 
\end{equation}
in which
\begin{equation}
\begin{array}{lllll}
D_\xi\Delta(z,m\tau/l,0) &=& x_{\xi}(m\tau;z,m\tau /l,0)-I&=&\Pi^m(z)-I \\  D_T\Delta(z,m\tau/l,0)&=& f(z)+x_T(m\tau;z,m\tau /l,0)&=&f(z)\\ D_\varepsilon\Delta(z,m\tau/l,0)&=&x_{\varepsilon}(m\tau;z,m\tau /l,0)&=&\chi(z)
\end{array} ,
\end{equation}
where we have denoted by $x_{\kappa}(m\tau;z,m\tau l,0)$ the solution of the first variational problem in the parameter $\kappa$ at time $m\tau$. The solution of the first variation in the period is zero since the period dependence of the vector field only appears at $O(\varepsilon)$. Exploiting Proposition \ref{prop:fact}, we project $\Delta_1$ using the invertible matrix $R_{ext}(z)$ defined as
\begin{equation}
R_{ext}(z)=\begin{bmatrix}
R_r(z)& 0 \\ -f^*(z) & 0 \\ 0 &1 \\ DH(z)&0 
\end{bmatrix} ,
\end{equation}
in order to obtain
\begin{equation}
\label{eq:LSRsplit}
\begin{array}{c}
\displaystyle \Delta_{1}'(\eta,\varepsilon)=R_{ext}(z)\Delta_1(\eta,\varepsilon)=\begin{pmatrix} \Delta_r(\delta,\sigma,s,\varepsilon) \\ \Delta_c(\delta,\sigma,s,\varepsilon) \end{pmatrix}=\begin{cases} A(z)\begin{pmatrix} \delta \\ \sigma \end{pmatrix} +b(z)\\
\langle DH(z),\chi(z) \rangle \end{cases} +O(\varepsilon) , \\ \\
A(z)=\begin{bmatrix}
A_r(z)& 0 & 0\\ w^*(z) & m\tau_v & -1 \\ D_\xi L(z,m\tau)K_r(z) & \langle D_\xi L(z,m\tau),v(z) \rangle & lD_T L(z,m\tau) 
\end{bmatrix} , \\ \\
b(z)=\begin{pmatrix}
R_r(z)\chi(z) \\  -\langle f(z),\chi(z) \rangle \\ 0 \\ 
\end{pmatrix} .
\end{array}
\end{equation}
We now aim to show that $A(z)$ is an invertible matrix for any $s$. Due to its block matrix structure, its determinant reads
\begin{equation}
\mathrm{det}\big(A(z)\big)=\mathrm{det}\big(A_r(z)\big)\big(\langle D_\xi L(z,m\tau),v(z)\rangle+ml\tau_v  D_{\tau}L(z,m\tau)\big) ,
\end{equation}
where the first factor is nonzero due to Proposition \ref{prop:fact}. For the second factor, we use the identity
\begin{equation}
L(x_0(t;\vartheta_{\mathcal{S}}(\lambda)),m\mathcal{T}(\lambda))= \lambda ,
\end{equation}
whose differentiation in $\lambda$, evaluation $\lambda=\lambda_z$ and division by $||D\vartheta_{\mathcal{S}}(\lambda_z)||$ yields
\begin{equation}
\langle D_\xi L(z,m\tau),v(z)\rangle+ml\tau_v  D_{T}L(z,m\tau)=1/||D\vartheta_{\mathcal{S}}(\lambda_z)|| ,
\end{equation} 
proving that $A(z)$ is then invertible. Hence, we can solve for $\delta$ and $\sigma$ in the leading order term of $\Delta_r$ for any $s$, so that the implicit function theorem ($\Delta_1'\in C^{r-1}$ with $r\geq 2$) guarantees that we can locally express $\delta=\delta_r(s,\varepsilon)$ and $\sigma=\sigma_r(s,\varepsilon)$ such that $\Delta_r(\delta_r(s,\varepsilon),\sigma_r(s,\varepsilon),s,\varepsilon)=0$. Thus, we have shown that the perturbed solutions of $\Delta_{l,L}(\xi,T,\varepsilon)=0$ have a one-to-one correspondence with the zeros of the bifurcation function $B^{m:l}_L(s,\varepsilon)$ defined as
\begin{equation}
B^{m:l}_L:\mathbb{R}\times \mathbb{R} \rightarrow \mathbb{R}, \,\,\,\,\,\,  B^{m:l}_L(s,\varepsilon)=\Delta_c(\delta_r(s,\varepsilon),\sigma_r(s,\varepsilon),s,\varepsilon)=M^{m:l}(s)+O(\varepsilon) ,
\end{equation}
where $M^{m:l}(s) = \langle DH(z),\chi(z) \rangle $. Moreover, this function does not depend on the mapping $L$ used as a constraint in Eq. (\ref{eq:DisplMapC}).

We now aim to simplify the Melnikov-type function $M^{m:l}(s)$ to the form in Eq. (\ref{eq:MelFun}). Denoting with $Y(t;x_0(s;p))$ the solution of the first variational problem for the vector field $f(x)$, we write explicitly the solution of the first variational problem in $\varepsilon$ (see \cite{Chicone2000}) leading to
\begin{equation}
\begin{array}{rl}
M^{m:l}(s) =& \langle \, DH(x_0(s;p))\, ,\, x_{\varepsilon}(m\tau;x_0(s;p),m\tau/l,0) \rangle \\ \\
=& \displaystyle \big\langle \, DH(x_0(s;p))\, , \, Y(m\tau;x_0(s;p)) \cdot \\ \\ & \displaystyle \cdot \,\int_0^{m\tau}Y^{-1}(t;x_0(s;p))\, g(x_0(t;x_0(s;p)),t;m\tau/l,0) dt \, \big\rangle ,
\end{array} 
\end{equation}
and we recall that the dynamics on an energy surface $H(x)=H(p)$ (that acts as a codim. 1 invariant manifold), is characterised by (see Proposition 3.2 in \cite{Li2000} for a proof)
\begin{equation}
\label{eq:imandyn}
DH(x_0(t+s;p))=DH(x_0(s;p))Y^{-1}(t+s;p) , \,\,\,\,\,\,  DH(p)Y(m\tau;p)=DH(p).
\end{equation}
Equation (\ref{eq:imandyn}) leads to
\begin{equation}
\begin{array}{ll}
M^{m:l}(s)&\displaystyle=\big\langle \, DH(x_0(s;p)) \, ,\, \int_0^{m\tau}Y^{-1}(t+s;p)\, g(x_0(t+s;p),t;m\tau/l,0) dt \, \big\rangle \\ \\
&\displaystyle=\int_0^{m\tau} \big\langle\, DH(x_0(s;p))\, ,\,Y^{-1}(t+s;p) \, g(x_0(t+s;p),t;m\tau/l,0) \,\big\rangle dt \\ \\
&\displaystyle=\int_0^{m\tau} \big\langle\, DH(x_0(t+s;p))\, , \,  g(x_0(t+s;p),t;m\tau/l,0)\,\big\rangle dt ,
\end{array} 
\end{equation} 
and this function is clearly smooth and $m\tau$-periodic. 
\end{proof} 

\subsection{Proof of Theorem \ref{thm:stpert}}
\label{app:A1stpert}
\begin{proof}
Thanks to Theorem \ref{thm:redthm}, we are able to reduce the persistence problem of Eq. (\ref{eq:DisplMapC}) to the study of $B^{m:l}_L(s,\varepsilon)$. The zeros of this function mark the existence of periodic orbits for $\varepsilon$ small enough which smoothly connect to $\mathcal{Z}$ at $\varepsilon=0$. Note that, if the $M^{m:l}(s)\equiv 0$, then no conclusions for persistence can be drawn solely from $M^{m:l}(s)$. Indeed, we need to analyse the $O(\varepsilon)$ term in $B^{m:l}_L$.

If the Melnikov function remains bounded away from zero, then we conclude the last statement thanks to the fact that no zeros of the bifurcation function exists for $\varepsilon$ small enough.

We now analyse the case of simple zeros. Assuming that the conditions in Eq. (\ref{eq:transvzero}) hold for $s_0$, the implicit function theorem guarantees that we can express $s=s(\varepsilon)$ from the bifurcation function $B^{m:l}_L(s,\varepsilon)$. According to the proof of Theorem \ref{thm:redthm}, we can define 
\begin{equation}
\delta(\varepsilon)=\delta_r(s(\varepsilon),\varepsilon), \,\,\,\,\,\, \sigma(\varepsilon)=\sigma_r(s(\varepsilon),\varepsilon), \,\,\,\,\,\, \eta(\varepsilon)=(\delta(\varepsilon),\sigma(\varepsilon),s(\varepsilon)),
\end{equation}
such that $\Delta_{l,L}(\varphi(\eta(\varepsilon),\varepsilon),\varepsilon)=0$ for a sufficiently small neighbourhood $C_0\subset\mathbb{R}$. Hence, we can express the initial conditions and the periods
\begin{equation}
\begin{cases}
\xi (\varepsilon) = x_0(s(\varepsilon);p)+\varepsilon K_{\mathbb{T}\mathcal{Z}}\big(x_0(s(\varepsilon);p)\big)\delta(\varepsilon)=x_0(s_0;p) +O(\varepsilon)\\ 
lT(\varepsilon)=m\tau + \varepsilon \sigma(\varepsilon) =m\tau +O(\varepsilon) 
\end{cases}
\end{equation}
of periodic orbits solving system (\ref{eq:NAutSysG}) and satisfying $L(\xi (\varepsilon),lT(\varepsilon))=\lambda$ for small enough $\varepsilon>0$.

Finally, the second statement of Theorem \ref{thm:stpert} is a direct consequence of the intermediate value theorem. Namely, the existence of a simple zero for the Melnikov function implies that $ M^{m:l}(s)$ is not constant and there exist points $s_1=s_0-\epsilon$ and $s_2=s_0+\epsilon$ such that $M^{m:l}(s_1) M^{m:l}(s_2)<0$ for $\epsilon>0$ small enough. Due to periodicity, we also have $ M^{m:l}(s_2) M^{m:l}(s_1+m\tau)<0$. Thus, there exists at least another $\hat{s}_0\in (s_2,\, s_1+m\tau)$ such that $M^{m:l}(\hat{s}_0)= 0$ due Bolzano's theorem and it must be a zero at which the function changes sign, i.e., a topologically transverse zero. 
\end{proof}
\begin{remark}
\label{rmk:smoothpers}
Theorem \ref{thm:stpert} guarantees smooth persistence only. There may be degenerate cases where there exist periodic orbits of system (\ref{eq:NAutSysG}) that are still $O(\varepsilon)$-close to $\mathcal{Z}$, but they can only be continuously connected to the latter or not connected at all. The Melnikov function in (\ref{eq:MelFun}) cannot prove existence of such orbits.
\end{remark}
\begin{remark}
\label{rmk:melnder}
To analyse the type of zeros of the Melnikov function, it is convenient to have closed formulae for its derivatives. The first derivative can be computed as
\begin{equation}
DM^{m:l}(s)=-\int_0^{m\tau} \big\langle\, DH(x_0(t+s;p))\, , 
\, \partial_t g(x_0(t+s;p),t;\tau m/l,0)\,\big\rangle dt ,
\end{equation}
given that
\begin{equation}
\begin{array}{ll}
\displaystyle DM^{m:l}(s)&=\displaystyle \int_0^{m\tau} D_s \big\langle\, DH\, , 
\, g\,\big\rangle dt 
=\displaystyle \int_0^{m\tau} D_t \big\langle\, DH\, , 
\, g\,\big\rangle dt  - \int_0^{m\tau} \big\langle\, DH\, , 
\, \partial_t g\,\big\rangle dt \\ \\
&=-\displaystyle \int_0^{m\tau} \big\langle\, DH\, , 
\, \partial_t g\,\big\rangle dt ,
\end{array}
\end{equation}
which is again a smooth periodic function. Thus, a transverse zero $s_0$ of $M^{m:l}(s)$ must satisfy:
\begin{equation}
\begin{array}{l}
\displaystyle \int_0^{m\tau} \big\langle\, DH(x_0(t+s_0;p))\, , 
\, g(x_0(t+s_0;p),t;\tau m/l,0)\,\big\rangle dt=0 , \\ \\
\displaystyle \int_0^{m\tau} \big\langle\, DH(x_0(t+s_0;p))\, , 
\, \partial_t g (x_0(t+s_0;p),t;\tau m/l,0)\,\big\rangle dt \neq0 .
\end{array}
\end{equation}
Assuming enough smoothness, the second derivative of $M^{m:l}(s)$ is likewise
\begin{equation}
D^2_{ss} M^{m:l}(s)=\int_0^{m\tau} \big\langle\, DH(x_0(t+s;p))\, , 
\, \partial^2_{tt} g (x_0(t+s;p),t;\tau m/l,0)\,\big\rangle dt 
\end{equation}
A similar formula follows for high-order derivatives.
\end{remark}
\begin{remark}
\label{rmk:fnhyp}
Note that if the orbit family can be parametrised with the period, one can directly insert the exact resonance condition into the displacement map. In this case, the method developed in \cite{ChiconeR2000} applies in a straightforward way in what the authors call a non-degenerate case. Compared with the discussion in that reference, we simplified the final Melnikov function.
\end{remark}

\subsection{Proof of Theorem \ref{thm:ftpert}}
\label{app:A1ftpert}
Once the reduction to a scalar bifurcation function has been performed as in Theorem \ref{thm:redthm}, the statements in Theorem \ref{thm:ftpert} follow from results of the bifurcation analysis outlined in \cite{GolubitskySchaeffer1985}. Specifically, one can look at Theorem 2.1 and Table 2.3 in Chapter IV to recognise the bifurcation problem. In that reference, the singular bifurcation isola birth is called \textit{isola centre}.

We further remark that a saddle-node bifurcation persists in the perturbed setting. Indeed, defining
\begin{equation}
B_{sn}(s,\kappa,\varepsilon)=\begin{cases}
B^{m:l}_L(s,\kappa,\varepsilon) \\
D_s B^{m:l}_L(s,\kappa,\varepsilon) 
\end{cases} ,
\end{equation}
we find that
\begin{equation}
B_{sn}(s_0,\kappa_0,0)=0 , \,\,\,\,\,\, \mathrm{det}\big( D_{s,\kappa} B_{sn}(s_0,\kappa_0,0)\big) = -D^2_{ss} M^{m:l} (s_0,\kappa_0) D_{\kappa}M^{m:l}(s_0,\kappa_0) \neq 0 .
\end{equation}
Therefore, the implicit function theorem applies, guaranteeing that a locally unique orbit persists at $s_{sn}=s_0+O(\varepsilon)$, $\kappa_{sn}=\kappa_0+O(\varepsilon)$.

\section{Melnikov function with monoharmonic, space-independent forcing}
\label{app:A2}
In this Appendix, we show the derivations that lead to Eq. (\ref{eq:WPFAD}). By substituting the Fourier series of Eq. (\ref{eq:Fseries}) into (\ref{eq:WPFAD}), we find that
\begin{equation}
\begin{array}{rl}
w^{m:l}(s,e)=&-\displaystyle e\sum_{k=1}^\infty  \int_0^{m\tau }k\omega  \langle a_k ,f_e\rangle  
\sin\left( k\omega  (t+s)\right)\cos\left( \frac{l}{m} \omega t\right)dt+ \\ \\
&\displaystyle +e\sum_{k=1}^\infty \int_0^{m\tau }k\omega  \langle b_k ,f_e\rangle  
\cos\left( k\omega  (t+s)\right)\cos\left( \frac{l}{m} \omega t\right)dt .
\end{array} 
\end{equation}
Expanding using trigonometric addition formulae, we obtain
\begin{equation}
\label{eq:Wsimp}
\begin{array}{rl}
w^{m:l}(s,e)=&-\displaystyle e\sum_{k=1}^\infty k\omega  \langle a_k ,f_e\rangle  \cos\left( k\omega  s \right)
\int_0^{m\tau }  \sin\left( k\omega t \right)\cos\left( \frac{l}{m} \omega t\right)dt+ \\ \\
&-\displaystyle e\sum_{k=1}^\infty k\omega  \langle a_k ,f_e\rangle  \sin\left( k\omega s \right)
\int_0^{m\tau }  \cos\left( k\omega t \right)\cos\left( \frac{l}{m} \omega t\right)dt+ \\ \\
&+\displaystyle e\sum_{k=1}^\infty k\omega  \langle b_k ,f_e\rangle  \cos\left( k\omega  s\right)
\int_0^{m\tau }  \cos\left( k\omega t \right)\cos\left( \frac{l}{m} \omega t\right)dt+ \\ \\
&-\displaystyle e\sum_{k=1}^\infty k\omega  \langle b_k ,f_e\rangle  \sin\left( k\omega  s \right)
\int_0^{m\tau }  \sin\left( k\omega t \right)\cos\left( \frac{l}{m} \omega t\right)dt
\end{array} .
\end{equation}
We recall the following trigonometric integral identities with $k\neq j$ :
\begin{equation}
\label{eq:trigid}
\begin{array}{c}
\displaystyle \int_0^{\tau}   \sin\left( k\omega t \right)\cos\left( j\omega t\right)=\int_0^{\tau}   \sin\left( k\omega t \right)\sin\left( j\omega t\right)=\int_0^{\tau}   \cos\left( k\omega t \right)\cos\left( j\omega t\right)=0 , \\ \\
\displaystyle \int_0^{\tau}   \sin\left( k\omega t \right)\cos\left( k\omega  t\right)=0 , \,\,\,\,\,\, \int_0^{\tau} \sin^2\left( k\omega t \right)=\int_0^{\tau}   \cos^2\left( k\omega t \right)=\frac{\tau} {2} .
\end{array}
\end{equation}
Thus, the integrals in the first and last summations in Eq. (\ref{eq:Wsimp}) are always zero. We first discuss the case $m\neq 1$. We call $m\tau =\tau_o$ so that Eq. (\ref{eq:Wsimp}) becomes
\begin{equation}
\label{eq:Wsimpm}
\begin{array}{rl}
w^{m:l}(s,e)=&-\displaystyle e\sum_{k=1}^\infty k\omega  \langle a_k ,f_e\rangle  \sin\left( k\omega  s\right)
\int_0^{\tau_o}  \cos\left( \frac{2km\pi}{\tau_o}t \right)\cos\left( \frac{2l\pi}{\tau_o} t\right)dt+ \\ \\
&+\displaystyle e\sum_{k=1}^\infty k\omega  \langle b_k ,f_e\rangle  \cos\left( k\omega  s\right)
\int_0^{\tau_o}  \cos\left( \frac{2km\pi}{\tau_o}t \right)\cos\left( \frac{2l\pi}{\tau_o} t\right)dt .
\end{array} 
\end{equation}
Therefore, to obtain nonzero integrals in Eq. (\ref{eq:Wsimpm}), we need that $km=l$ according to Eq. (\ref{eq:trigid}). However, since we choose $l$ and $m$ to be positive integers and relatively prime, that condition will never hold. We then conclude that $w^{m:l}(s,e)\equiv 0$ for $m\neq 1$.

For $m=1$, only the terms for $k=l$ can be nonzero in Eq. (\ref{eq:Wsimpm}), resulting in
\begin{equation}
\label{eq:Wsimp1}
w^{m:l}(s)= -l\pi \langle a_l ,f_e\rangle  \sin\left( l\omega s \right)+l\pi\langle b_l ,f_e\rangle  \cos\left( l\omega s \right) .
\end{equation}
Thus, we recover Eq. (\ref{eq:WPFAD}) with the proper definitions of $A_{l,e}$ and $\alpha_{l,e}$.
\end{appendix}

\bibliographystyle{unsrt}
\bibliography{Biblio}

\begin{thebibliography}{10}

\bibitem{Rosenberg1962}
R.M. Rosenberg.
\newblock The normal modes of nonlinear n-degree-of-freedom systems.
\newblock {\em Journal of Applied Mechanics}, 29:7 -- 14, 1962.

\bibitem{Vakakis1997}
A.F. Vakakis.
\newblock Non-linear normal modes, ({NNM}s) and their application in vibration
  theory: an overview.
\newblock {\em Mechanical Systems and Signal Processing}, 11(1):3 -- 22, 1997.

\bibitem{Vakakis2001}
A.F. Vakakis, editor.
\newblock {\em Normal Modes and Localization in Nonlinear Systems}.
\newblock Springer, Dordrecht, 2001.

\bibitem{Vakakis2008}
A.F. Vakakis, L.I. Manevitch, Y.V. Mikhlin, V.N. Pilipchuk, and A.A. Zevin.
\newblock {\em Normal Modes and Localization in Nonlinear Systems}.
\newblock Wiley Blackwell, 1 2008.

\bibitem{Avramov2011}
K.V. Avramov and Y.V. Mikhlin.
\newblock Nonlinear normal modes for vibrating mechanical systems. review of
  theoretical developments.
\newblock {\em ASME Applied Mechanics Reviews}, 63(6), 2011.

\bibitem{Avramov2013}
K.V. Avramov and Y.V. Mikhlin.
\newblock Review of applications of nonlinear normal modes for vibrating
  mechanical systems.
\newblock {\em ASME Applied Mechanics Reviews}, 65(2), 2013.

\bibitem{Kerschen2014}
G.~Kerschen.
\newblock {\em Modal Analysis of Nonlinear Mechanical Systems}, volume 555 of
  {\em CISM International Centre for Mechanical Sciences}.
\newblock Springer-Verlag Wien, 2014.

\bibitem{NayfehM2007}
A.H. Nayfeh and D.T. Mook.
\newblock {\em Nonlinear Oscillations}.
\newblock Wiley, 2007.

\bibitem{Peeters2011b}
M.~Peeters, G.~Kerschen, and J.C. Golinval.
\newblock Modal testing of nonlinear vibrating structures based on nonlinear
  normal modes: experimental demonstration.
\newblock {\em Mechanical Systems and Signal Processing}, 25(4):1227 -- 1247,
  2011.

\bibitem{Szalai2017}
R.~Szalai, D.~Ehrhardt, and G.~Haller.
\newblock Nonlinear model identification and spectral submanifolds for
  multi-degree-of-freedom mechanical vibrations.
\newblock {\em Proceedings of the Royal Society of London A: Mathematical,
  Physical and Engineering Sciences}, 473(2202), 2017.

\bibitem{Touze2006}
C.~Touz{é} and M.~Amabili.
\newblock Nonlinear normal modes for damped geometrically nonlinear systems:
  Application to reduced-order modelling of harmonically forced structures.
\newblock {\em Journal of Sound and Vibration}, 298(4):958 -- 981, 2006.

\bibitem{Sombroek2018}
C.S.M. Sombroek, P.~Tiso, L.~Renson, and G.~Kerschen.
\newblock Numerical computation of nonlinear normal modes in a modal derivative
  subspace.
\newblock {\em Computers \& Structures}, 195:34 -- 46, 2018.

\bibitem{Polunin2016}
P.~M. {Polunin}, Y.~{Yang}, M.~I. {Dykman}, T.~W. {Kenny}, and S.~W. {Shaw}.
\newblock Characterization of mems resonator nonlinearities using the ringdown
  response.
\newblock {\em Journal of Microelectromechanical Systems}, 25(2):297--303,
  April 2016.

\bibitem{Carpineto2014}
N.~Carpineto, W.~Lacarbonara, and F.~Vestroni.
\newblock Hysteretic tuned mass dampers for structural vibration mitigation.
\newblock {\em Journal of Sound and Vibration}, 333(5):1302 -- 1318, 2014.

\bibitem{Mojahed2018}
A.~Mojahed, K.~Moore, L.A. Bergman, and A.F. Vakakis.
\newblock Strong geometric softening-hardening nonlinearities in an oscillator
  composed of linear stiffness and damping elements.
\newblock {\em International Journal of Non-Linear Mechanics}, 107:94 -- 111,
  2018.

\bibitem{Liu2019}
Y.~Liu, A.~Mojahed, L.A. Bergman, and A.F. Vakakis.
\newblock A new way to introduce geometrically nonlinear stiffness and damping
  with an application to vibration suppression.
\newblock {\em Nonlinear Dynamics}, 96(3):1819--1845, May 2019.

\bibitem{Habib2018}
G.~Habib, G.I. Cirillo, and G.~Kerschen.
\newblock Isolated resonances and nonlinear damping.
\newblock {\em Nonlinear Dynamics}, 93(3):979--994, 2018.

\bibitem{Hill2016}
T.L. Hill, S.A. Neild, and A.~Cammarano.
\newblock An analytical approach for detecting isolated periodic solution
  branches in weakly nonlinear structures.
\newblock {\em Journal of Sound and Vibration}, 379:150 -- 165, 2016.

\bibitem{Ponsioen2019}
S.~Ponsioen, T.~Pedergnana, and G.~Haller.
\newblock Analytic prediction of isolated forced response curves from spectral
  submanifolds.
\newblock {\em Nonlinear Dynamics}, Jun 2019.

\bibitem{Hill2017}
T.L. Hill, A.~Cammarano, S.A. Neild, and D.A.W. Barton.
\newblock Identifying the significance of nonlinear normal modes.
\newblock 473(2199), 2017.

\bibitem{Sanders2007}
J.A. Sanders, F.~Verhulst, and J.~Murdock.
\newblock {\em Averaging Methods in Nonlinear Dynamical Systems}, volume~59 of
  {\em Applied Mathematical Sciences}.
\newblock Springer-Verlag New York, 2 edition, 2007.

\bibitem{Touze2004}
C.~Touzé, O.~Thomas, and A.~Chaigne.
\newblock Hardening/softening behaviour in non-linear oscillations of
  structural systems using non-linear normal modes.
\newblock {\em Journal of Sound and Vibration}, 273(1):77 -- 101, 2004.

\bibitem{Neild2011}
S.A. Neild and D.J. Wagg.
\newblock Applying the method of normal forms to second-order nonlinear
  vibration problems.
\newblock {\em Proceedings of the Royal Society of London A: Mathematical,
  Physical and Engineering Sciences}, 467(2128):1141--1163, 2011.

\bibitem{Hill2015}
T.L. Hill, A.~Cammarano, S.A. Neild, and D.J. Wagg.
\newblock Interpreting the forced responses of a two-degree-of-freedom
  nonlinear oscillator using backbone curves.
\newblock {\em Journal of Sound and Vibration}, 349:276 -- 288, 2015.

\bibitem{Vakakis2018}
A.F. Vakakis and A.~Blanchard.
\newblock Exact steady states of the periodically forced and damped duffing
  oscillator.
\newblock {\em Journal of Sound and Vibration}, 413:57--65, 1 2018.

\bibitem{Haller2016}
G.~Haller and S.~Ponsioen.
\newblock Nonlinear normal modes and spectral submanifolds: existence,
  uniqueness and use in model reduction.
\newblock {\em Nonlinear Dynamics}, 86(3):1493--1534, 2016.

\bibitem{Breunung2018}
T.~Breunung and G.~Haller.
\newblock Explicit backbone curves from spectral submanifolds of forced-damped
  nonlinear mechanical systems.
\newblock {\em Proceedings of the Royal Society of London A: Mathematical,
  Physical and Engineering Sciences}, 474(2213), 2018.

\bibitem{Ponsioen2018}
S.~Ponsioen, T.~Pedergnana, and G.~Haller.
\newblock Automated computation of autonomous spectral submanifolds for
  nonlinear modal analysis.
\newblock {\em Journal of Sound and Vibration}, 420:269 -- 295, 2018.

\bibitem{Lyapunov1992}
A.M. Lyapunov.
\newblock The general problem of the stability of motion.
\newblock {\em International Journal of Control}, 55(3):531 -- 534, 1992.

\bibitem{Renson2016a}
L.~Renson, G.~Kerschen, and B.~Cochelin.
\newblock Numerical computation of nonlinear normal modes in mechanical
  engineering.
\newblock {\em Journal of Sound and Vibration}, 364:177 -- 206, 2016.

\bibitem{Peeters2009}
M.~Peeters, R.~Viguié, G.~Sérandour, G.~Kerschen, and J.-C. Golinval.
\newblock Nonlinear normal modes, part {II}: toward a practical computation
  using numerical continuation techniques.
\newblock {\em Mechanical Systems and Signal Processing}, 23(1):195 -- 216,
  2009.
\newblock Special Issue: Non-linear Structural Dynamics.

\bibitem{Grolet2012}
A.~Grolet and F.~Thouverez.
\newblock On a new harmonic selection technique for harmonic balance method.
\newblock {\em Mechanical Systems and Signal Processing}, 30:43 -- 60, 2012.

\bibitem{Dankowicz2013}
H.~Dankowicz and F.~Schilder.
\newblock {\em Recipes for Continuation}.
\newblock Society for Industrial and Applied Mathematics, 2013.

\bibitem{Peeters2011a}
M.~Peeters, G.~Kerschen, and J.C. Golinval.
\newblock Dynamic testing of nonlinear vibrating structures using nonlinear
  normal modes.
\newblock {\em Journal of Sound and Vibration}, 330(3):486 -- 509, 2011.

\bibitem{Ehrhardt2016}
David~A. Ehrhardt and Matthew~S. Allen.
\newblock Measurement of nonlinear normal modes using multi-harmonic stepped
  force appropriation and free decay.
\newblock {\em Mechanical Systems and Signal Processing}, 76-77:612 -- 633,
  2016.

\bibitem{Peter2018}
S.~Peter, M.~Scheel, M.~Krack, and R.I. Leine.
\newblock Synthesis of nonlinear frequency responses with experimentally
  extracted nonlinear modes.
\newblock {\em Mechanical Systems and Signal Processing}, 101:498 -- 515, 2018.

\bibitem{Renson2016b}
L.~Renson, A.~Gonzalez-Buelga, D.A.W. Barton, and S.A. Neild.
\newblock Robust identification of backbone curves using control-based
  continuation.
\newblock {\em Journal of Sound and Vibration}, 367:145 -- 158, 2016.

\bibitem{Kerschen2009}
G.~Kerschen, M.~Peeters, J.C. Golinval, and A.F. Vakakis.
\newblock Nonlinear normal modes, part {I}: A useful framework for the
  structural dynamicist.
\newblock {\em Mechanical Systems and Signal Processing}, 23(1):170 -- 194,
  2009.

\bibitem{GH1983}
J.~Guckenheimer and P.J. Holmes.
\newblock {\em Nonlinear Oscillations, Dynamical Systems, and Bifurcations of
  Vector Fields}, volume~42 of {\em Applied Mathematical Sciences}.
\newblock Springer-Verlag New York, 1983.

\bibitem{MO2017}
K.R. Meyer and D.C. Offin.
\newblock {\em Introduction to Hamiltonian Dynamical Systems and the N-body
  Problem}, volume~90 of {\em Applied Mathematical Sciences}.
\newblock Springer-Verlag New York, 3 edition, 2017.

\bibitem{Malkin1949}
I.G. Malkin.
\newblock On poincar\'e's theory of periodic solutions.
\newblock {\em Akad. Nauk SSSR. Prikl. Mat. Meh.}, 13:633 -- 646, 1949.

\bibitem{Loud1959}
W.S. Loud.
\newblock Periodic solutions of a perturbed autonomous system.
\newblock {\em Annals of Mathematics}, 70(3):490 -- 529, 1959.

\bibitem{Farkas1994}
M.~Farkas.
\newblock {\em Periodic Motions}, volume 104 of {\em Applied Mathematical
  Sciences}.
\newblock Springer-Verlag New York, 1994.

\bibitem{MoserZehnder2005}
J.~Moser and E.J. Zehnder.
\newblock {\em Notes on Dynamical Systems}, volume~12 of {\em Courant Lecture
  Notes}.
\newblock American Mathematical Society, 2005.

\bibitem{Poincare1892}
H.~Poincaré.
\newblock {\em Les Méthodes Nouvelles de la Mécanique Céleste}.
\newblock Gauthier-Villars et Fils, Paris, 1892.

\bibitem{Arnold1964}
V.I. Arnol’d.
\newblock Instability of dynamical systems with many degrees of freedom.
\newblock {\em Dokl. Akad. Nauk SSSR}, 156(1):9 -- 12, 1964.

\bibitem{Melnikov1963}
V.K. Melnikov.
\newblock On the stability of a center for time-periodic perturbations.
\newblock {\em Tr. Mosk. Mat. Obs.}, 12:3 -- 52, 1963.

\bibitem{Yagasaki1996}
K.~Yagasaki.
\newblock The melnikov theory for subharmonics and their bifurcations in forced
  oscillations.
\newblock {\em SIAM Journal on Applied Mathematics}, 56(6):1720--1765, 1996.

\bibitem{Veerman1985}
P.~Veerman and P.~Holmes.
\newblock The existence of arbitrarily many distinct periodic orbits in a two
  degree of freedom hamiltonian system.
\newblock {\em Physica D: Nonlinear Phenomena}, 14(2):177--192, 1985.

\bibitem{Veerman1986}
P.~Veerman and P.~Holmes.
\newblock Resonance bands in a two degree of freedom hamiltonian system.
\newblock {\em Physica D: Nonlinear Phenomena}, 20(2):413 -- 422, 1986.

\bibitem{Yagasaki1999}
K.~Yagasaki.
\newblock Periodic and homoclinic motions in forced, coupled oscillators.
\newblock {\em Nonlinear Dynamics}, 20(4):319 -- 359, 1999.

\bibitem{Kunze2000}
M.~Kunze.
\newblock {\em Non-Smooth Dynamical Systems}, volume 1744 of {\em Lecture Notes
  in Mathematics}.
\newblock Springer-Verlag Berlin Heidelberg, 2000.

\bibitem{Shaw1989a}
S.W. Shaw and R.H. Rand.
\newblock The transition to chaos in a simple mechanical system.
\newblock {\em International Journal of Non-Linear Mechanics}, 24(1):41 -- 56,
  1989.

\bibitem{Shaw1989b}
J.~{Shaw} and S.W. {Shaw}.
\newblock {The onset of chaos in a two-degree-of-freedom impacting system}.
\newblock {\em Journal of Applied Mechanics}, 56:168, 1989.

\bibitem{Chicone1994}
C.~Chicone.
\newblock Lyapunov-{S}chmidt reduction and {M}elnikov integrals for bifurcation
  of periodic solutions in coupled oscillators.
\newblock {\em Journal of Differential Equations}, 112(2):407 -- 447, 1994.

\bibitem{Chicone1995}
C.~Chicone.
\newblock A geometric approach to regular prturbation theory with an
  application to hydrodynamics.
\newblock {\em Transactions of the American Mathematical Society}, 12(2):4559
  -- 4598, 1995.

\bibitem{ChiconeR2000}
M.B.H. Rhouma and C.~Chicone.
\newblock On the continuation of periodic orbits.
\newblock {\em Methods and Applications of Analysis}, 7(1):85 -- 104, 2000.

\bibitem{Buica2012}
A.~Buic\u{a}, J.~Llibre, and O.~Makarenkov.
\newblock Bifurcations from nondegenerate families of periodic solutions in
  lipschitz systems.
\newblock {\em Journal of Differential Equations}, 252(6):3899 -- 3919, 2012.

\bibitem{Chicone2000}
C.~Chicone.
\newblock {\em Ordinary Differential Equations with Applications}, volume~34 of
  {\em Texts in Applied Mathematics}.
\newblock Springer-Verlag New York, 1982.

\bibitem{Sepulchre1997}
J.A. Sepulchre and R.S. MacKay.
\newblock Localized oscillations in conservative or dissipative networks of
  weakly coupled autonomous oscillators.
\newblock {\em Nonlinearity}, 10(3):679, 1997.

\bibitem{Doedel2003}
F.J. Muñoz-Almaraz, E.~Freire, J.~Galán, E.~Doedel, and A.~Vanderbauwhede.
\newblock Continuation of periodic orbits in conservative and hamiltonian
  systems.
\newblock {\em Physica D: Nonlinear Phenomena}, 181(1):1 -- 38, 2003.

\bibitem{ChowHale1982}
S.N. Chow and J.K. Hale.
\newblock {\em Methods of Bifurcation Theory}, volume 251 of {\em Grundlehren
  der mathematischen Wissenschaften}.
\newblock Springer-Verlag New York, 1982.

\bibitem{GolubitskySchaeffer1985}
M.~Golubitsky and S.~Schaeffer.
\newblock {\em Singularities and Groups in Bifurcation Theory}, volume~51 of
  {\em Applied Mathematical Sciences}.
\newblock Springer-Verlag New York, 1985.

\bibitem{Govaerts2000}
W.J.F. Govaerts.
\newblock {\em Numerical Methods for Bifurcations of Dynamical Equilibria}.
\newblock Society for Industrial and Applied Mathematics, 2000.

\bibitem{Teschl2012}
G.~Teschl.
\newblock {\em Ordinary Differential Equations and Dynamical Systems}, volume
  140 of {\em Graduate Studies in Mathematics}.
\newblock American Mathematical Society, 2012.

\bibitem{Li2000}
M.Y. Li and J.S. Muldowney.
\newblock Dynamics of differential equations on invariant manifolds.
\newblock {\em Journal of Differential Equations}, 168(2):295 -- 320, 2000.

\end{thebibliography}

\end{document}